\documentclass[a4paper, 10pt]{article}

\usepackage[english]{babel}
\usepackage[utf8]{inputenc}

\usepackage[T1]{fontenc}
\usepackage{lmodern}

\usepackage{amsmath, amssymb, amsfonts, mathtools, amsthm}
\usepackage{url}
\usepackage{authblk}

\usepackage{tikz}
\usetikzlibrary{arrows}

\newtheorem{lemma}{Lemma}

\providecommand{\abs}[1]{\lvert#1\rvert}
\providecommand{\norm}[1]{\lVert#1\rVert}

\newcommand{\R}{\mathbb{R}}
\newcommand{\N}{\mathbb{N}}

\newcommand{\x}{\mathbf{x}}
\newcommand{\y}{\mathbf{y}}

\newcommand{\Id}{\mathrm{Id}}

\newcommand{\Tr}{\mathrm{Tr}}

\DeclareMathOperator*{\mean}{mean}

\newcommand{\resfolder}{figures/small}

\begin{document}

\title{A Numerical Framework for Efficient Motion Estimation on Evolving Sphere-Like Surfaces based on Brightness and Mass Conservation Laws}

\author{Lukas F. Lang}
\affil{\footnotesize Department of Applied Mathematics and Theoretical
Physics, University of Cambridge, Wilberforce Road, Cambridge CB3 0WA, United Kingdom}

\date{}
\maketitle

\begin{abstract}
\noindent
In this work we consider brightness and mass conservation laws for motion estimation on evolving Riemannian 2-manifolds that allow for a radial parametrisation from the 2-sphere.
While conservation of brightness constitutes the foundation for optical flow methods and has been generalised to said scenario, we formulate in this article the principle of mass conservation for time-varying surfaces which are embedded in Euclidean 3-space and derive a generalised continuity equation.
The main motivation for this work is efficient cell motion estimation in time-lapse (4D) volumetric fluorescence microscopy images of a living zebrafish embryo.
Increasing spatial and temporal resolution of modern microscopes require efficient analysis of such data.
With this application in mind we address this need and follow an emerging paradigm in this field: dimensional reduction.
In light of the ill-posedness of considered conservation laws we employ Tikhonov regularisation and propose the use of spatially varying regularisation functionals that recover motion only in regions with cells.
For the efficient numerical solution we devise a Galerkin method based on compactly supported (tangent) vectorial basis functions.
Furthermore, for the fast and accurate estimation of the evolving sphere-like surface from scattered data we utilise surface interpolation with spatio-temporal regularisation.
We present numerical results based on aforementioned zebrafish microscopy data featuring fluorescently labelled cells.
\end{abstract}

\section{Introduction}

Recent advances in microscopy imaging techniques allow to study cellular dynamics of biological model organisms in more detail than ever before, see e.g. \cite{KelSchmSanKhaBao10, KelSchmWitSte08, KrzGunSauStrHuf12}.
Time-lapse volumetric (4D) image sequences of the development of entire living animals can be captured in high resolution and on a sub-cellular scale.
However, increasing spatial and temporal resolutions require additional efforts in dealing with the resulting large volumes of data.
The need for efficient methods to analyse such data has already been acknowledged and is considered a major interdisciplinary challenge \cite{Kel13, ReyPeyHuiTom14}.

One promising approach in dealing with image sequences of this type is \emph{dimensional reduction}.
A geometric model of the observed organism is introduced and the captured data is considered only with respect to this geometry, see \cite{HeeStr15, SchmShaScheWebThi13}.
These efforts focus on the true shape---or an approximation---of the specimen and thereby reduce the spatial dimension of the data by considering only the restriction, or a suitable projection, to this geometry.
Due to the spatial sparsity of the volumetric data the essential information is preserved.

A major gain of this approach is that it can also reduce the computational effort during analysis of the recorded material, see e.g. \cite{KirLanSch13, KirLanSch14, KirLanSch15, LanSch17, SchmShaScheWebThi13}.
In addition, introducing a geometric representation of the specimen allows to compute accurate measurements, such as distances, on curved surfaces rather than in---possibly distorting---projections.
For the quantitative analysis of cellular processes this leads to a considerable improvement, see \cite{HeeStr15}.

The zebrafish is a popular and well-established animal research model that can be observed \emph{in vivo}.
Understanding its developmental process is of major interest.
We refer to \cite{KimBalKimUllSchi95} for a detailed discussion and illustrations.
Cellular dynamics of \emph{endodermal} cells are crucial for organ and tissue formation during early development of the organism.
Despite its importance, there is a lack of understanding of their migration and proliferation patterns \cite{AmaLemMosMcDWan14, SchmShaScheWebThi13}.
However, endodermal cells are known to form a so-called \emph{monolayer}, meaning that they do not stack on top of each other but rather float side by side forming a contiguous single-cell layer \cite{WarNus99}.
For the purpose of observation, these cells can be fluorescently labelled and recorded separately from the background by means of confocal laser-scanning microscopy.
Figure~\ref{fig:raw} illustrates a section of a captured image sequence containing only the upper hemisphere of the embryo.
Shown are nuclei of endodermal cells during the gastrula period forming a round surface in a single-cell layer.

\begin{figure}[t]
	\includegraphics[width=0.32\textwidth]{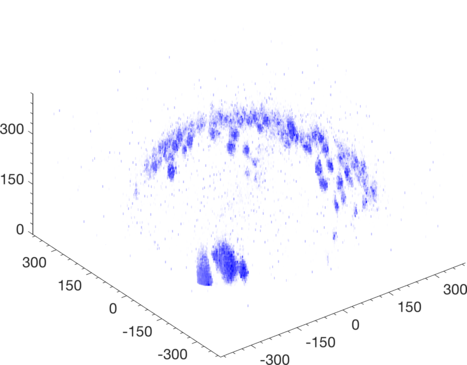} \hfill
	\includegraphics[width=0.32\textwidth]{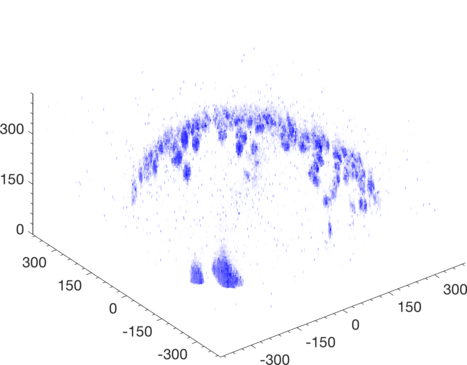} \hfill
	\includegraphics[width=0.32\textwidth]{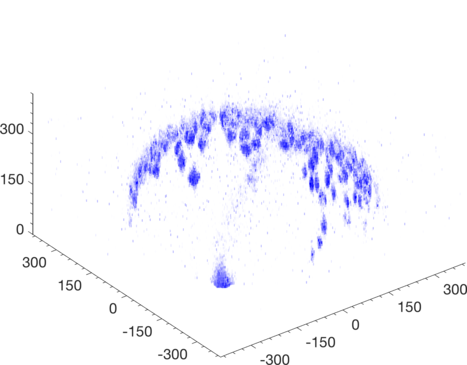} \\
	\includegraphics[width=0.32\textwidth]{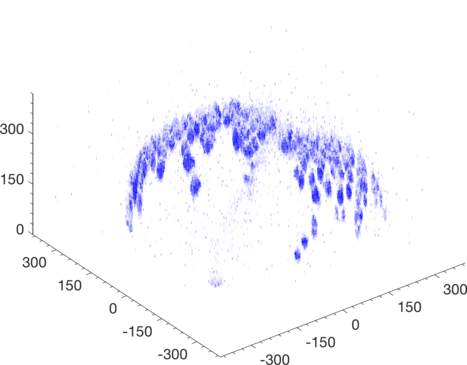} \hfill
	\includegraphics[width=0.32\textwidth]{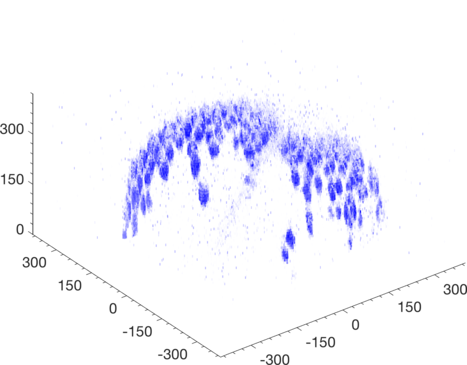} \hfill
	\includegraphics[width=0.32\textwidth]{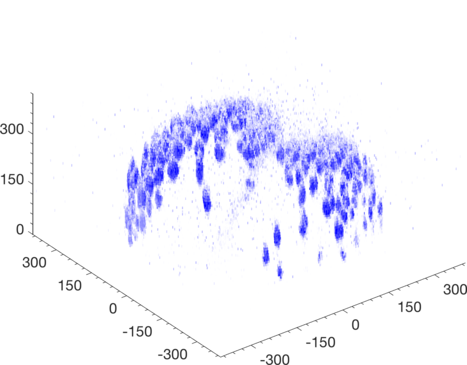} \\
	\caption{Frames $110, 120, \dots, 150$ (left to right, top to bottom) of a volumetric zebrafish microscopy image sequence recorded during early embryogenesis. The sequence contains 151 frames recorded at intervals of $120 \, \mathrm{s}$. Blue colour indicates fluorescence response. As time evolves, the initially spherical yolk develops a clearly visible dent, which is where the embryonic axis forms and cells eventually converge to, see also \cite[Figs.~11 and 15]{KimBalKimUllSchi95}. All dimensions are in micrometer ($\mu$m).}
	\label{fig:raw}
\end{figure}

\begin{figure}[t]
	\includegraphics[width=0.32\textwidth]{\resfolder/raw3-cxcr4aMO2_290112-150} \hfill
	\includegraphics[width=0.32\textwidth]{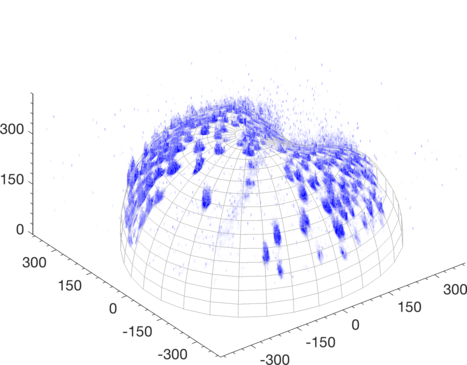} \hfill
	\includegraphics[width=0.32\textwidth]{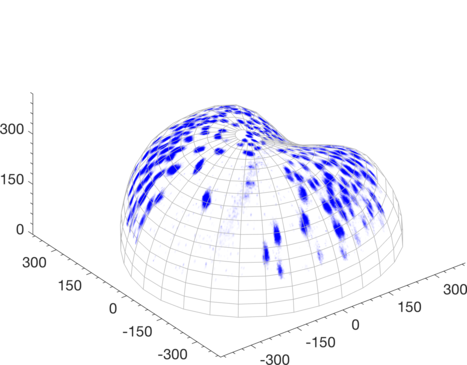}
	\caption{Frame no. 150 of the zebrafish image sequence. The left image depicts the unprocessed volumetric microscopy data $f^{\delta}$. The curved mesh in the center image illustrates a sphere-like surface fitted to approximate cell centres. The right image shows the surface data $\hat{f}$ obtained by taking the radial maximum intensity projection of $f^{\delta}$ onto the surface within a narrow band. For details see Sec.~\ref{sec:experiments}. All dimensions are in micrometer ($\mu$m).}
	\label{fig:approach}
\end{figure}

The primary goal of this article is quantitative motion estimation of endodermal cells in fluorescence microscopy data of a living zebrafish embryo.
Efficient motion estimation is crucial for the large-scale automated analysis of such datasets and can provide new insights into cellular mechanisms and the dynamic behaviour of cells.
See e.g. \cite{AmaMyeKel13, BorOriVieWhi13, MelCamLomRizVer07, QueMenCam10, SchmShaScheWebThi13}.

We build upon previous work \cite{LanSch17} where the deforming single-cell layer is modelled as a closed surface $\mathcal{M}_{t} \subset \R^{3}$, $t \in [0, T]$, of the form
\begin{equation*}
	\left\{ \tilde{\rho}(t, x) x: x \in \mathcal{S}^{2} \right\}
\end{equation*}
together with a time-dependent function $\hat{f}(t, \cdot): \mathcal{M}_{t} \to \R$ that indicates fluorescence response and is assumed to be directly proportional to the observed intensity.
Here, $\tilde{\rho}(t, \cdot): \mathcal{S}^{2} \to (0, \infty)$ is a radial deformation of the 2-sphere $\mathcal{S}^{2}$.
See Fig.~\ref{fig:approach} for the general idea and Fig.~\ref{fig:surfaces} for a sketch.

The main idea, which was developed in \cite{KirLanSch13, KirLanSch15}, is to conceive the  motion of a cell---as it migrates through Euclidean 3-space---only with respect to this moving surface.
As a consequence, the velocity $\mathbf{\hat{U}}(t, x) \in \R^{3}$ of a cell which always stays on this surface can be decomposed into the sum of a---prescribed and in general not tangential---surface velocity $\mathbf{\hat{V}}(t, x) \in \R^{3}$ and a purely tangential velocity $\mathbf{\hat{w}}(t, x) \in T_{x}\mathcal{M}_{t}$ which is relative to $\mathbf{\hat{V}}$.
Here, $T_{x} \mathcal{M}_{t}$ denotes the tangent space at $x \in \mathcal{M}_{t}$.
See Fig.~\ref{fig:sketch} for illustration.
In further consequence, one can estimate $\mathbf{\hat{w}}$ from the data $\hat{f}$ by solving a parametrised optical flow problem
\begin{equation*}
	d_{t}^{\mathbf{\hat{V}}} \hat{f} + \nabla_{\mathcal{M}} \hat{f} \cdot \mathbf{\hat{w}} = 0
\end{equation*}
on this evolving surface.
Here, $d_{t}^{\mathbf{\hat{V}}}$ denotes a suitable temporal derivative, $\nabla_{\mathcal{M}}$  the (spatial) surface gradient, and dot the standard inner product.
As a result, the velocity of a cell can be estimated as $\mathbf{\hat{U}} = \mathbf{\hat{V}} + \mathbf{\hat{w}}$.
While $\mathbf{\hat{w}}$ is relative to the chosen $\mathbf{\hat{V}}$ and should be interpreted with care, it is reasonable to assume that their sum is close to the true velocity of a cell.
Integral curves then yield approximate cell trajectories.

In this model, $\hat{f}$ is assumed to satisfy a \emph{brightness constancy assumption}, which is typical for optical flow-based motion estimation: the intensity $\hat{f}$ is conserved along trajectories of moving points.
However, in many situations it is too restrictive and possibly violated, see e.g. the discussion in \cite[Sec.~3]{CorMemPer02}.

In this article, we address this issue and assume that $\hat{f}$ instead fulfils \emph{conservation of mass}.
We derive a suitable generalisation of the continuity equation to evolving surfaces which are embedded in $\R^{3}$ and obtain the pointwise conservation law
\begin{equation*}
	d_{t}^{\mathbf{\hat{N}}} \hat{f} + \nabla_{\mathcal{M}} \cdot (\hat{f} \mathbf{\hat{u}}) - \hat{f}KV = 0.
\end{equation*}
Here, $d_{t}^{\mathbf{\hat{N}}}$ denotes the normal time derivative and $\nabla_{\mathcal{M}} \cdot$ the surface divergence, $K$ is related to surface curvature, and $V$ is the scalar normal velocity of the moving surface.
The main advantage, compared to \cite{KirLanSch13, KirLanSch15, LanSch17}, is that one is able to directly infer the entire tangential velocity $\mathbf{\hat{u}} = \mathrm{P}_{\mathcal{M}}(\mathbf{\hat{U}})$ of cells from the data $\hat{f}$,
where $\mathrm{P}_{\mathcal{M}}$ denotes the orthogonal projector onto the tangent space of $\mathcal{M}$.
The normal component of $\mathbf{\hat{U}}$ is prescribed by the surface's normal velocity and the total velocity $\mathbf{\hat{U}}$ can thus be estimated by adding the tangential part $\mathbf{\hat{u}}$.

In view of the ill-posedness of above-mentioned conservation equations, we follow a variational approach and minimise a Tikhonov-type functional of the form
\begin{equation*}
	\mathcal{D}(\cdot, \hat{f}) + \alpha \mathcal{R}(\cdot),
\end{equation*}
where $\mathcal{D}$ is the squared $L^{2}$ norm of the left hand side of one of the above identities, $\mathcal{R}$ is a regularisation functional, and $\alpha > 0$ a parameter balancing the two terms.
One of the major advantages of applying this energy to mass preservation is that it favours regularity of $\mathbf{\hat{u}}$ rather than regularity of $\mathbf{\hat{w}}$, which depends on the prescribed---and in practice often unknown---surface velocity.

We address in this work another major point.
While dense motion estimation
is often desired for complex natural scenes, it is redundant for aforementioned microscopy data.
Such data are considerably simpler due to the characteristic shape of cell nuclei, the absence of occlusions, and their sparsity.
In order to mitigate undesired fill-in effects of quadratic regularisation functionals to areas where $\hat{f}$ is zero, we introduce novel regularisation functionals inspired by image segmentation models, see e.g. \cite{BarChaChuJunKir11}.
Given a segmentation of the cells, motion is only estimated in regions where data is present.

\subsection{Contributions}

The contributions of this article are as follows.
First, we discuss and introduce brightness and mass conservation laws on evolving surfaces, and relate these two concepts to each other.
While conservation of brightness is the foundation for the optical flow equation and has been dealt with in \cite{KirLanSch13, KirLanSch15},  we generalise in this article the principle of mass conservation to time-varying surfaces which are embedded in Euclidean 3-space and derive a generalised continuity equation.
For numerical convenience we devise a parametrised version thereof.

Second, we propose new spatially varying regularisation functionals for motion estimation based on the discussed conservation laws.
They are specially tailored to mentioned fluorescence microscopy data and indicate motion only in regions with cells present.

Third, for the numerical solution we propose a Galerkin method based on compactly supported (vectorial) basis functions.
Resulting sparsity effects lead to vast improvements in performance compared to \cite{LanSch17}, which uses globally supported basis functions.
Moreover, we provide a formula for the Hilbert-Schmidt norm of the covariant derivative of a vector field, which is commonly used for tangent vector field regularisation.
As a result, the Gram-Schmidt orthonormalisation of the tangent basis is rendered redundant, yielding another major performance gain compared to \cite{BauGraKir15, KirLanSch15, LanSch17}.

Fourth, for extracting a sphere-like surface together with surface image data from aforementioned microscopy image sequences we propose surface interpolation with spatio-temporal regularisation.
Compared to \cite{LanSch17}, where only spatial regularisation is used, this leads to a more accurate estimation of the surface's (normal) velocity and the surface data, which---in turn---should improve the accuracy of the computed cell velocities.

Fifth, we present numerical results based on aforementioned zebrafish microscopy data.
We compute and compare cell motion estimated by imposing either of the two discussed conservation laws.

\subsection{Related Work}

Concerning dense motion estimation in $\R^{2}$, Horn and Schunck \cite{HorSchu81} were the first to propose a variational approach based on conservation of brightness.
They suggested minimising a Tikhonov-type functional with $H^{1}$ Sobolev semi-norm regularisation, favouring spatially regular vector fields.
For a general introduction to the topic see e.g. \cite{AubDerKor99, AubKor06} and for a survey on various optical flow functionals see \cite{WeiBruBroPap06}.
Well-posedness of the Horn-Schunck functional was proved by Schnörr \cite{Schn91a}, where the problem was treated on irregular planar domains and solved by means of a finite element method.

Weickert and Schnörr \cite{WeiSchn01b} proposed an extension to the domain $[0, T] \times \R^{2}$.
The model includes spatial as well as temporal first-order regularisation and is particularly appealing whenever integral curves are to be recovered.
A framework unifying various spatial and temporal regularisers was established by the same authors in \cite{WeiSchn01a}.
For the comparison of different motion estimation methods an evaluation framework was developed in \cite{BakSchaLewRotBla11}.

Only recently, generalisations to non-Euclidean and non-static domains have received increasing attention.
For the purpose of robot vision, optical flow on the static round sphere was considered in \cite{ImiSugTorMoc05, TorImiSugMoc05}.
With an application to brain image analysis, Lef\`{e}vre and Baillet \cite{LefBai08} generalised the Horn-Schunck functional to static surfaces which are embedded in $\R^{3}$ and proved well-posedness.
Numerically, the problem was solved on a triangle mesh with a finite element method.

With the aim of analysing cell motion in fluorescence microscopy data, Kirisits~et~al. \cite{KirLanSch13, KirLanSch15} considered a generalisation of the Horn-Schunck functional to evolving surfaces with boundary.
In particular, in \cite{KirLanSch15} the authors proposed a generalisation of the spatio-temporal model in \cite{WeiSchn01b}.
Minimisation was performed by solving the associated Euler-Lagrange equations in the coordinate domain with a finite-difference scheme.
In \cite{KirLanSch14}, they studied several  decomposition models for optical flow on the static 2-sphere.
The problems were solved by means of projection to finite-dimensional spaces spanned by tangent vector spherical harmonics.

In Bauer~et~al. \cite{BauGraKir15}, optical flow on moving manifolds with and without spatial boundary was investigated.
The authors considered product manifolds for which an appropriate Riemannian metric was constructed and well-posedness of their formulation was shown.

In Lang and Scherzer \cite{LanSch17}, the embryo of a zebrafish was modelled as an evolving sphere-like surface.
The generalised optical flow problem was rewritten as an equivalent problem on the 2-sphere and solved by means of a Galerkin method based on tangent vector spherical harmonics.
In order to find the sphere-like surface from microscopy data, surface interpolation from approximate cell centres was proposed.

We also refer to \cite{AmaMyeKel13, MelCamLomRizVer07, QueMenCam10}, where the optical flow was computed to track cells in microscopy 
data, and to \cite{BorOriVieWhi13}, where the optical flow was utilised to infer the motion of neural crest cells in zebrafish microscopy images.
Moreover, in Schmid~et~al. \cite{SchmShaScheWebThi13}, the sphere was used to model the embryo of a zebrafish and the motion of endodermal cells was computed in map projections by means of fluid image registration.

According to \cite{CorMemPer02}, Schunck \cite{Schu84} was the first to propose motion estimation in image sequences based on the continuity equation.
Since then, it has been utilised in numerous works, as mass preservation is a particularly appealing alternative for fluid motion estimation.
For instance, in \cite{BerHerYou00, CorMemPer02, ZhoKamGol00} it is was used to analyse meteorological satellite images.
In \cite{WilAmaLanLeu00}, fluid flow was estimated from image sequences and in \cite{SonLea91}, the continuity equation was utilised to find the 3D deformation of a beating human heart in tomography images.
Moreover, in \cite{Amin94} it was used to analyse blood flow and in \cite{CorHeiArrMemSan06} fluid flow was estimated by means of an integrated continuity equation paired with second-order regularisation.
In \cite{DawBruJiaButBur10}, they proposed to use the continuity equation for cardiac motion correction of 3D images obtained by positron emission tomography.
See also \cite{HeiMemSchn10} for a survey of variational methods for fluid flow estimation.
Finally, in \cite{Dir15}, for the purpose of joint motion estimation and image reconstruction, both the optical flow and the continuity equation were used.

The remainder of this article is structured as follows.
In Sec.~\ref{sec:background}, we introduce sphere-like evolving surfaces and their basic properties.
Moreover, we discuss vectorial Sobolev spaces on manifolds, and introduce compactly supported (vectorial) basis functions and scalar spherical harmonics.
In Sec.~\ref{sec:model}, we discuss brightness and mass conservation on evolving surfaces and introduce for each conservation principle a variational formulation.
Section~\ref{sec:numerics} is dedicated to their numerical solution.
We derive necessary and sufficient conditions, which are evaluated and solved on the 2-sphere.
In order to find a sphere-like surface from real microscopy data, we propose surface interpolation with spatio-temporal regularisation and discuss its numerical solution by means of scalar spherical harmonic expansion.
In Sec.~\ref{sec:experiments} we discuss, compare, and visualise numerical results based on aforementioned microscopy data of a zebrafish.
Finally, Sec.~\ref{sec:conclusion} concludes the article.
\section{Notation and Background} \label{sec:background}

\subsection{Sphere-Like Evolving Surfaces} \label{sec:surfaces}

Let us denote by $\mathcal{S}^{2} = \{x \in \R^{3}: \norm{x} = 1\}$ the 2-sphere embedded in the 3-dimensional Euclidean space and let $\norm{x} = \sqrt{x \cdot x}$ denote the norm of $\R^{n}$, $n = \{2, 3\}$.
Moreover, we denote by
\begin{equation*}
	\x: \Omega \subset \R^{2} \to \R^{3}
\end{equation*}
a regular parametrisation of $\mathcal{S}^{2}$ mapping points $\xi = (\xi^{1}, \xi^{2})^{\top} \in \Omega$ in the coordinate domain to points $x = (x^{1}, x^{2}, x^{3})^{\top} \in \mathcal{S}^{2}$ on the sphere.
The outward unit normal at $x \in \mathcal{S}^{2}$ is denoted by $\mathbf{\tilde{N}}(x)$.

Let $I \coloneqq [0, T] \subset \R$ denote a time interval.
We consider a family $\mathcal{M} = \{ \mathcal{M}_{t}\}_{t \in I}$ of closed smooth 2-manifolds $\mathcal{M}_{t} \subset \R^{3}$ and assume that each $\mathcal{M}_{t}$ is regular and oriented by the outward unit normal field $\mathbf{\hat{N}}(t, x) \in \R^{3}$, $x \in \mathcal{M}_{t}$.
Furthermore, we assume that $\mathcal{M}$ admits a smooth and smoothly evolving parametrisation of the form
\begin{equation}
	\y: I \times \Omega \to \R^{3}, \quad (t, \xi^{1}, \xi^{2})^{\top} \mapsto \tilde{\rho}(t, \x(\xi^{1}, \xi^{2})) \x(\xi^{1}, \xi^{2}) \in \mathcal{M}_{t}
\label{eq:param}
\end{equation}
with $\tilde{\rho}: I \times \mathcal{S}^{2} \to (0, \infty)$ being a sufficiently smooth (radius) function.
We refer to $\mathcal{M}$ as \emph{evolving sphere-like surface}.

Let us denote by $\hat{f}: \mathcal{M} \to \R$ a smooth function on $\mathcal{M}$, by $f: I \times \Omega \to \R$ its coordinate representation, and by $\tilde{f}: I \times \mathcal{S}^{2} \to \R$ its representation on $\mathcal{S}^{2}$, respectively.
For $t \in I$ and $\xi \in \Omega$, they are related by
\begin{equation}
	f(t, \xi) = \tilde{f}(t, \x(\xi)) = \hat{f}(t, \y(t, \xi)).
\label{eq:coordrepr}
\end{equation}
The partial derivative with respect to $\xi^{i}$ is abbreviated by $\partial_{i}$.
Accordingly, the partial derivative with respect to time is denoted by $\partial_{t}$.
At this point let us clarify further notational conventions.
Functions or vector fields for which the domain is $\mathcal{S}^{2}$ will be indicated with a tilde and functions for which the domain is $\mathcal{M}$ will be indicated with a hat.
Their corresponding coordinate representation is treated without special indication.

Moreover, we define a smooth (spatial) extension $\bar{f}$ of $\hat{f}$ to $\R \setminus \{ 0 \}$ which is constant along radial lines.
It is given by
\begin{equation}
	\bar{f}(t, x) = \hat{f} \left(t, \tilde{\rho} \left(t, \frac{x}{\norm{x}} \right) \frac{x}{\norm{x}} \right).
\label{eq:extension}
\end{equation}
Figure~\ref{fig:surfaces} illustrates the setting.

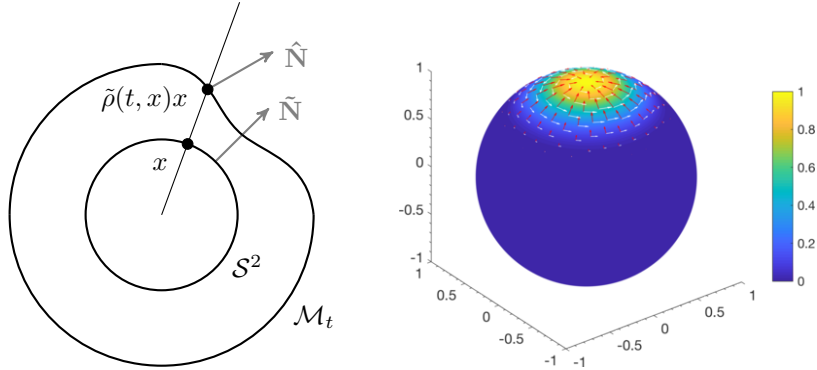
\begin{figure}[t]
	\begin{center}
	\begin{tikzpicture}
		\draw [-stealth', thick, gray] (canvas polar cs:angle=45,radius=1cm) to +(canvas polar cs:angle=45,radius=1cm);
		\draw [-stealth', thick, gray] (canvas polar cs:angle=70,radius=1.77cm) to +(canvas polar cs:angle=30,radius=1cm);
		\draw [thick] (0, 0) circle (1);
		\draw [thick] (0, 2) .. controls (canvas polar cs:angle=70,radius=2.05cm) and (canvas polar cs:angle=65,radius=1.55cm) .. (canvas polar cs:angle=45,radius=1.5cm);
		\draw [thick] (canvas polar cs:angle=45,radius=1.5cm) .. controls (canvas polar cs:angle=30,radius=1.6cm) and (canvas polar cs:angle=20,radius=2.05cm) .. (2, 0);
		\draw [thick] (-2, 0) .. controls (-2, 1.11) and (-1.11, 2) .. (0, 2);
		\draw [thick] (-2, 0) .. controls (-2, -1.11) and (-1.11, -2) .. (0, -2) .. controls (1.11, -2) and (2, -1.11) .. (2, 0);
		\draw (0, 0) to (canvas polar cs:angle=70,radius=3cm);
		\filldraw [black] (canvas polar cs:angle=70,radius=1cm) circle (2pt);
		\filldraw [black] (canvas polar cs:angle=70,radius=1.77cm) circle (2pt);
		\node [left] at (canvas polar cs:angle=73,radius=0.7cm) {$x$};
		\node [left] at (canvas polar cs:angle=73,radius=1.6cm) {$\tilde{\rho}(t, x) x$};
		\node [gray, right] at (canvas polar cs:angle=55,radius=2.6cm) {$\mathbf{\hat{N}}$};
		\node [gray, right] at (canvas polar cs:angle=45,radius=2cm) {$\mathbf{\tilde{N}}$};
		\node [right] at (canvas polar cs:angle=-40,radius=1.05cm) {$\mathcal{S}^2$};
		\node [right] at (canvas polar cs:angle=-40,radius=2.1cm) {$\mathcal{M}_{t}$};
	\end{tikzpicture}
	\hspace{0.5cm}
	\includegraphics[width=0.45\textwidth]{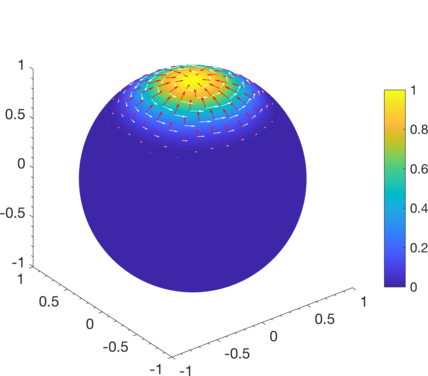}
	\end{center}
	\caption{Left: Illustration of a cut through surfaces $\mathcal{S}^{2}$ and $\mathcal{M}_{t}$ intersecting the origin. The extension $\bar{f}(t, \cdot)$ is constant along the shown radial line. Surface normals are depicted in grey. Right: Vectorial basis functions $\mathbf{\tilde{y}}_{j}^{(1)}$ (red) and $\mathbf{\tilde{y}}_{j}^{(2)}$ (white), and the corresponding zonal function $\tilde{b}_{j}$ centred at the north pole $x_{j} = (0, 0, 1)^{\top}$. The parameters of the corresponding function $b_{h}^{(k)}$ were chosen as $k = 3$ and $h = 0.6$, cf. Sec~\ref{sec:background:basisfun}.}
	\label{fig:surfaces}
\end{figure}

In the following, let us consider time $t \in I$ arbitrary but fixed.
We denote the tangent plane at a point $\y(t, \xi) \in \mathcal{M}_{t}$ by $T_{\y(t, \xi)}\mathcal{M}_{t}$ and the tangent bundle by $T\mathcal{M}_{t} = \bigl\{ \{\y(t, \xi)\} \times T_{\y(t, \xi)}\mathcal{M}_{t}: \xi \in \Omega \bigr\}$.
The orthogonal projector onto $T_{x}\mathcal{M}_{t}$, $x \in \mathcal{M}_{t}$, is denoted by $\mathrm{P}_{\mathcal{M}}$ and is given by
\begin{equation}
	\mathrm{P}_{\mathcal{M}}(t, x) = \Id - \mathbf{\hat{N}}(t, x) \mathbf{\hat{N}}(t, x)^{\top} \in \R^{3 \times 3}.
\label{eq:orthogonalprojector}
\end{equation}

The set
\begin{equation}
	\{ \partial_{1} \y(t, \xi), \partial_{2} \y(t, \xi) \}
\label{eq:coordbasis}
\end{equation}
forms a basis for the tangent space $T_{\y(t, \xi)}\mathcal{M}_{t}$ at $\y(t, \xi) \in \mathcal{M}_{t}$.
As a consequence, a tangent vector $\mathbf{\hat{v}} \in T_{\y(t, \xi)}\mathcal{M}_{t}$ can uniquely be represented as $\mathbf{\hat{v}} = \sum_{i=1}^{2} v^{i} \partial_{i} \y$, with $\mathbf{v} = (v^{1}, v^{2})^{\top} \in \R^{2}$ being its coordinate representation.
The elements $v^{i}$ are called components of $\mathbf{\hat{v}}$.
For a tangent vector field $\mathbf{\hat{v}}$ we define its smooth extension $\mathbf{\bar{v}}$ to $\R \setminus \{ 0 \}$ component-wise and analogous to~\eqref{eq:extension}.

In the following we will use Einstein summation convention and sum over each index letter appearing exactly twice in an expression, one as a sub- and once as a superscript.
For instance, we will write $\mathbf{\hat{v}} = v^{i} \partial_{i}\y$ for the sake of brevity.
As a further notational convection, boldface letters are used to denote vector fields.
In particular, lower case boldface letters refer to tangent vector fields, whereas upper case boldface letters refer to general vector fields in $\R^{3}$, with the exception of the parametrisations $\x$ and $\y$.
Moreover, we will drop arguments, such as $(t, \xi)$ or $(t, x)$, whenever clear from the context.

The elements of \eqref{eq:coordbasis} form the gradient matrix
\begin{equation*}
\begin{aligned}
	D\y & = \begin{pmatrix}
		\partial_{1} \y & \partial_{2} \y
	\end{pmatrix} \\
	& = \begin{pmatrix}
		(\partial_{1} \rho)\x & (\partial_{2} \rho)\x
	\end{pmatrix} + \rho D\x \in \R^{3 \times 2},
\end{aligned}
\end{equation*}
with $D\x = \begin{pmatrix} \partial_{1} \x & \partial_{2} \x \end{pmatrix}$ being the gradient matrix associated with the parametrisation $\x$.
See \cite[Sec.~2.1]{LanSch17} for the derivation.
The positive definite matrix $(g_{ij}) = D\y^{\top}D\y$ is commonly referred to as \emph{Riemannian metric} and is given by
\begin{equation*}
	g = \begin{pmatrix}
		(\partial_{1} \rho)^{2} & \partial_{1} \rho \partial_{2} \rho \\
		\partial_{2} \rho \partial_{1} \rho & (\partial_{2} \rho)^{2}
	\end{pmatrix} + \rho^{2} D\x^{\top}D\x.
\end{equation*}
The elements of its inverse $g^{-1}$ are denoted by $(g^{ij})$.
Both are tensors and obey a transformation law when changing from one basis to another.
To this end, let $\{\mathbf{\hat{e}}_{1}, \mathbf{\hat{e}}_{2} \}$ be an arbitrary basis of $T_{x}\mathcal{M}_{t}$ at $x \in \mathcal{M}_{t}$ such that
\begin{equation}
	\mathbf{\hat{e}}_{i} = \alpha_{i}^{j} \partial_{j} \y.
\label{eq:changeofbasis}
\end{equation}
Moreover, let $(\alpha^{-1})_{i}^{j}$ be the inverse of the matrix $(\alpha_{i}^{j})$.
Then, for a $(p, q)$-tensor $T_{j_{1}, \dots, j_{q}}^{i_{1}, \dots, i_{p}}$ of order $p+q$ which is defined in the basis $\{ \partial_{1} \y, \partial_{2} \y \}$, its representation in the basis $\{\mathbf{\hat{e}}_{1}, \mathbf{\hat{e}}_{2} \}$ is given by
\begin{equation}
	\mathfrak{T}_{j_{1}', \dots, j_{q}'}^{i_{1}', \dots, i_{p}'} = (\alpha^{-1})_{i_{1}}^{i_{1}'} \dots (\alpha^{-1})_{i_{p}}^{i_{p}'} T_{j_{1}, \dots, j_{q}}^{i_{1}, \dots, i_{p}} \alpha_{j_{1}'}^{j_{1}} \dots \alpha_{j_{q}'}^{j_{q}}.
\label{eq:translaw}
\end{equation}
See e.g. \cite{Lee97} for details.

The surface gradient of a function $\hat{f}$, as given in \eqref{eq:coordrepr}, is defined by
\begin{equation*}
	\nabla_{\mathcal{M}} \hat{f} \coloneqq \mathrm{P}_{\mathcal{M}} \nabla_{\R^{3}} \bar{f},
\end{equation*}
where $\nabla_{\R^{3}}$ is the usual gradient of the embedding space and $\bar{f}$ is the extension defined in \eqref{eq:extension}.
In particular, for $\mathcal{M}_{t} = \mathcal{S}^{2}$ we have $\nabla_{\mathcal{S}^{2}} \tilde{f} = \nabla_{\R^{3}} \bar{f}$.
Moreover, for a tangent vector $\mathbf{\hat{v}} = v^{i} \partial_{i} \y \in T_{x} \mathcal{M}_{t}$, $x \in \mathcal{M}_{t}$, we have
\begin{equation}
	\nabla_{\mathcal{M}} \hat{f} \cdot \mathbf{\hat{v}} = v^{i} \partial_{i} f,
\label{eq:surfgradip}
\end{equation}
see \cite[Sec.~2.1]{LanSch17}.

For a function $\tilde{f}: \mathcal{S}^{2} \to \R$ we define the spherical Laplace-Beltrami in accordance to the surface gradient as
\begin{equation}
	\Delta_{\mathcal{S}^{2}} \tilde{f} = - \Delta_{\R^{3}} \bar{f},
\label{eq:laplacebeltrami}
\end{equation}
where $\Delta_{\R^{3}}$ denotes the Laplacian of $\R^{3}$ and $\bar{f}$ is the extension defined in \eqref{eq:extension}.

For an arbitrary surface $\mathcal{M}_{t} \subset \R^{3}$ embedded in 3-space, the total curvature, which is twice the mean curvature, is defined as
\begin{equation}
	K = - \nabla_{\R^{3}} \cdot \mathbf{\hat{N}}.
\label{eq:totalcurv}
\end{equation}
A numerically convenient representation is
\begin{equation*}
	K = \Tr\left( \begin{pmatrix}
		\partial_{11} \y \cdot \mathbf{\hat{N}} & \partial_{12} \y \cdot \mathbf{\hat{N}} \\
		\partial_{21} \y \cdot \mathbf{\hat{N}} & \partial_{22} \y \cdot \mathbf{\hat{N}}
	\end{pmatrix} g^{-1} \right),
\end{equation*}
where $\Tr$ denotes the trace of a matrix.
See \cite[Chap.~8]{Lee97} for details.

Naturally, the chosen parametrisation $\y$ of $\mathcal{M}_{t}$ admits a smooth map $\tilde{\phi}(t, \cdot): \mathcal{S}^{2} \to \mathcal{M}_{t}$ of the form
\begin{equation*}
	\tilde{\phi}(t, x): x \mapsto \tilde{\rho}(t, x)x.
\end{equation*}
The differential $D\tilde{\phi}(t, x): T_{x}\mathcal{S}^{2} \to T_{\tilde{\phi}(t, x)}\mathcal{M}_{t}$ of $\tilde{\phi}$ is a linear map given by
\begin{equation}
	D\tilde{\phi}(t, x) = \tilde{\rho}(t, x)\mathrm{Id} + x \nabla_{\mathcal{S}^{2}} \tilde{\rho}(t, x)^{\top} \in \R^{3 \times 3},
\label{eq:dphi}
\end{equation}
see \cite[Sec.~2.1]{LanSch17} for its derivation.
It provides a unique identification of a tangent vector field $\mathbf{\tilde{v}}$ on $\mathcal{S}^{2}$ with a tangent vector field $\mathbf{\hat{v}} = D\tilde{\phi}(\mathbf{\tilde{v}})$ on $\mathcal{M}_{t}$ and for $\mathbf{\tilde{v}} = v^{i} \partial_{i} \x$ we have $\mathbf{\hat{v}} = v^{i} \partial_{i} \y$.
In other words, the differential acts solely on the tangent basis, cf. \cite[Sec.~2.1]{LanSch17}.

For $t \in I$, the surface integral of a function $\hat{f}$, as defined in \eqref{eq:coordrepr}, is given by
\begin{equation}
	\int_{\mathcal{M}_{t}} \hat{f} \; d\mathcal{M}_{t} = \int_{\Omega} f J\y \; d\xi,
\label{eq:surfintegralcoord}
\end{equation}
where $(J\y)^{2} = \det(g)$ is the Jacobian of $\y$.
See Thm.~3 in \cite[p.~88]{DoC92}.
Moreover, by \cite[Lemma~2.1]{LanSch17}, we have
\begin{equation}
	\int_{\mathcal{M}_{t}} \hat{f} \; d\mathcal{M}_{t} = \int_{\mathcal{S}^{2}} \tilde{f} \tilde{\rho} \sqrt{\norm{\nabla_{\mathcal{S}^{2}} \tilde{\rho}}^{2} + \tilde{\rho}^{2}} \; d\mathcal{S}^{2}.
\label{eq:surfintegral}
\end{equation}

For further details on the concepts discussed above we refer the reader to standard differential geometry books, such as \cite{Car76, Car92, Lee97, Lee13}.

\subsection{Vectorial Sobolev Spaces on Manifolds} \label{sec:background:sobolevspaces}

In the following, we consider $t \in I$ and $x \in \mathcal{M}_{t}$ arbitrary but fixed.
Recall that $\mathbf{\bar{v}}$ denotes the component-wise extension~\eqref{eq:extension} of a tangent vector field $\mathbf{\hat{v}}$ on $\mathcal{M}_{t}$.
We define the covariant derivative of $\mathbf{\hat{v}}$ at a point $x \in \mathcal{M}_{t}$ along a tangent vector $\mathbf{\hat{u}} \in T_{x}\mathcal{M}_{t}$ as
\begin{equation*}
	\nabla_{\mathbf{\hat{u}}} \mathbf{\hat{v}}(t, x) \coloneqq \bigl( \mathrm{P}_{\mathcal{M}} \nabla_{\R^{3}} \mathbf{\bar{v}}(\mathbf{\hat{u}}) \bigr)(t, x).
\end{equation*}

In particular, for $\mathbf{\hat{u}} = \partial_{i} \y$ being an element of the coordinate basis \eqref{eq:coordbasis} and $\mathbf{\hat{v}} = v^{j} \partial_{j} \y$, it reads in terms of coordinates
\begin{equation*}
	\nabla_{\partial_{i} \y} \mathbf{\hat{v}} = \bigl( \partial_{i} v^{j} + v^{k} \Gamma_{ik}^{j} \bigr) \partial_{j} \y,
\end{equation*}
see e.g. \cite[Lemma 4.3]{Lee97}.
Here, $\Gamma_{ik}^{j}$ denote the Christoffel symbols with regard to the coordinate basis, that is, $\nabla_{\partial_{i} \y} \partial_{k} \y = \Gamma_{ik}^{j} \partial_{j} \y$.
Let us denote the above coefficients by
\begin{equation}
	D_{i}v^{j} \coloneqq \partial_{i} v^{j} + v^{k} \Gamma_{ik}^{j}.
\label{eq:covderivcoeff}
\end{equation}
They obey the usual tensorial transformation law \eqref{eq:translaw}, see. e.g. \cite[Lemma 4.7]{Lee97}.

Furthermore, the covariant derivative is a linear operator $\nabla \mathbf{\hat{v}}(t, x): T_{x}\mathcal{M}_{t} \to T_{x}\mathcal{M}_{t}$ and its Hilbert-Schmidt norm is given by
\begin{equation}
	\norm{\nabla \mathbf{\hat{v}}(t, x)}_{2}^{2} = \sum_{i=1}^{2} \norm{\nabla_{\mathbf{\hat{e}}_{i}} \mathbf{\hat{v}}(t, x)}^{2},
\label{eq:hsnorm}
\end{equation}
where $\{ \mathbf{\hat{e}}_{1}, \mathbf{\hat{e}}_{2} \}$ is an arbitrary orthonormal basis of the tangent space $T_{x}\mathcal{M}_{t}$.
We highlight that \eqref{eq:hsnorm} is invariant with regard to the chosen parametrisation $\y$.
The following lemma provides a convenient way for its computation:
\begin{lemma} \label{lem:covderiv}
Let $t \in I$ and $x = \y(t, \xi)$ be arbitrary for some $\xi \in \Omega$.
Then, for $\mathbf{\hat{v}} = v^{i} \partial_{i} \y$, it holds that
\begin{equation}
	\norm{\nabla \mathbf{\hat{v}}}_{2}^{2} = g_{k\ell} g^{ij} D_{i}v^{k} D_{j}v^{\ell},
\label{eq:covderiv}
\end{equation}
where we have omitted the arguments $(t, x)$ on the left-hand and $(t, \xi)$ on the right-hand side.
\end{lemma}
\begin{proof}
First, let us show that the right-hand side of \eqref{eq:covderiv} is parametrisation independent.
To this end, let $\{\partial_{1} \y, \partial_{2} \y\}$ and $\{\mathbf{\hat{e}}_{1}, \mathbf{\hat{e}}_{2}\}$ be arbitrary bases for $T_{x}\mathcal{M}_{t}$ such that its relation is given by \eqref{eq:changeofbasis}.
Then, by \cite[Lemma 4.7]{Lee97} and \eqref{eq:translaw}, we have the transformation law
\begin{equation*}
	\mathfrak{D}_{i}\mathfrak{v}^{k} = (\alpha^{-1})_{s}^{k} D_{t}v^{s} \alpha_{i}^{t}
\end{equation*}
for the components \eqref{eq:covderivcoeff} of the covariant derivative.
Moreover, $(g_{ij})$ and $(g^{ij})$ transform as $\mathfrak{g}_{k\ell} = g_{mn} \alpha_{k}^{m} \alpha_{\ell}^{n}$ and $\mathfrak{g}^{ij} = (\alpha^{-1})_{p}^{i} (\alpha^{-1})_{q}^{j} g^{pq}$, respectively.
Recall that, by definition, $\alpha_{i}^{k} (\alpha^{-1})_{k}^{j} = \delta_{i}^{j}$ and $(\alpha^{-1})_{i}^{k} \alpha_{k}^{j} = \delta_{i}^{j}$.
As a consequence,
\begin{equation*}
\begin{aligned}
	\mathfrak{g}_{k\ell} \mathfrak{g}^{ij} \mathfrak{D}_{i}\mathfrak{v}^{k} \mathfrak{D}_{j}\mathfrak{v}^{\ell} & = g_{mn} \alpha_{k}^{m} \alpha_{\ell}^{n} (\alpha^{-1})_{p}^{i} (\alpha^{-1})_{q}^{j} g^{pq} (\alpha^{-1})_{s}^{k} D_{t}v^{s} \alpha_{i}^{t} (\alpha^{-1})_{u}^{\ell} D_{w}v^{u} \alpha_{j}^{w} \\
	& = g_{mn} g^{pq} D_{t}v^{s} D_{w}v^{u} \delta_{s}^{m} \delta_{u}^{n} \delta_{p}^{t} \delta_{q}^{w} \\
	& = g_{mn} g^{pq} D_{p}v^{m} D_{q}v^{n}.
\end{aligned}
\end{equation*}
Suppose now that $\{\mathbf{\hat{e}}_{1}, \mathbf{\hat{e}}_{2}\}$ is orthonormal so that $\mathfrak{g} = (\delta_{ij})$ and $\mathfrak{g}^{-1} = (\delta^{ij})$.
Then,
\begin{equation*}
	\mathfrak{g}_{k\ell} \mathfrak{g}^{ij} \mathfrak{D}_{i}\mathfrak{v}^{k} \mathfrak{D}_{j}\mathfrak{v}^{\ell} = \sum_{i, k} (\mathfrak{D}_{i}\mathfrak{v}^{k})^{2} = \sum_{i} \norm{\nabla_{\mathbf{\hat{e}}_{i}} \mathbf{\hat{v}}}^{2} = \norm{\nabla \mathbf{\hat{v}}}_{2}^{2},
\end{equation*}
where the second equality follows from the fact that
\begin{equation*}
\begin{aligned}
	\sum_{i} \norm{\nabla_{\mathbf{\hat{e}}_{i}} \mathbf{\hat{v}}}^{2} & = \sum_{i} \mathfrak{D}_{i}\mathfrak{v}^{k} \mathbf{\hat{e}}_{k} \cdot \mathfrak{D}_{i}\mathfrak{v}^{\ell} \mathbf{\hat{e}}_{\ell} \\
	& = \sum_{i} \delta_{k\ell} \mathfrak{D}_{i}\mathfrak{v}^{k} \mathfrak{D}_{i}\mathfrak{v}^{\ell} \\
	& = \sum_{i, k} (\mathfrak{D}_{i}\mathfrak{v}^{k})^{2} \\
\end{aligned}
\end{equation*}
and the last equality is by definition \eqref{eq:hsnorm}.
Finally, the claim follows from combining these equations.
\end{proof}

We define for each $t \in I$ the Sobolev space $H^{1}(\mathcal{M}_{t}, T\mathcal{M}_{t})$ as the completion of $C^{\infty}(\mathcal{M}_{t}, T\mathcal{M}_{t})$ tangent vector fields with respect to
\begin{equation}
	\norm{\mathbf{\hat{v}}(t, \cdot)}_{H^{1}(\mathcal{M}_{t}, T\mathcal{M}_{t})}^{2} \coloneqq \int_{\mathcal{M}_{t}} \norm{\nabla \mathbf{\hat{v}}(t, x)}_{2}^{2} \; d\mathcal{M}_{t}.
\label{eq:h1norm}
\end{equation}
Let us add that \eqref{eq:h1norm} is a norm whenever $\mathcal{M}_{t}$ is diffeomorphic to the 2-sphere since, by virtue of the Hairy Ball Theorem, no covariantly constant tangent vector field but $\mathbf{\hat{v}} = 0$ exists, see e.g.~\cite[p.~125]{Hir94}.
We refer to \cite{Heb96, Tri92} for more details on Sobolev spaces on Riemannian manifolds.

For a tangent vector field $\mathbf{\hat{v}} = v^{i} \partial_{i} \y$, its surface divergence is defined as
\begin{equation}
	\nabla_{\mathcal{M}} \cdot \mathbf{\hat{v}} = \Tr(\nabla \mathbf{\hat{v}}) = \sum_{i = 1}^{2} \nabla_{\mathbf{\hat{e}}_{i}} \mathbf{\hat{v}} \cdot \mathbf{\hat{e}}_{i} = \mathfrak{D}_{i}\mathfrak{v}^{i},
\label{eq:surfdiv}
\end{equation}
where $\{\mathbf{\hat{e}}_{1}, \mathbf{\hat{e}}_{2}\}$ is an orthonormal basis of the tangent space and $\mathfrak{D}_{i}\mathfrak{v}^{i}$ is defined analogous to \eqref{eq:covderivcoeff}, see \cite{Lee97} for details.
\subsection{Compactly Supported Basis Functions} \label{sec:background:basisfun}

Let $h \in (0, 1)$ and let $k \in \N_{0} = \{0, 1, 2, \dots \}$.
Then, we define the one-dimensional piecewise polynomial function $b_{h}^{(k)}: [-1, 1] \to \R$ as
\begin{equation*}
	b_{h}^{(k)}(\tau) = \begin{cases}
		0 & \text{for } -1 \le \tau \le h, \\
		\frac{(\tau - h)^{k}}{(1-h)^{k}} & \text{for } h < \tau \le 1.
	\end{cases}
\end{equation*}
The parameter $h$ controls the support and $k$ is its degree.
For a point $x_{j} \in \mathcal{S}^{2}$ we define the $x_{j}$-zonal function
\begin{equation}
	\tilde{b}_{j}: \mathcal{S}^{2} \to \R, \quad x \mapsto b_{h}^{(k)}(x_{j} \cdot x),
\label{eq:zonalfunc}
\end{equation}
which, as a consequence, is compactly supported on $\mathcal{S}^{2}$.
See \cite{FreSchr95, Schr97} for further details.
Moreover, we define the tangent vector fields
\begin{equation}
\begin{aligned}
	\mathbf{\tilde{y}}_{j}^{(1)} & \coloneqq \nabla_{\mathcal{S}^{2}} \tilde{b}_{j}, \\
	\mathbf{\tilde{y}}_{j}^{(2)} & \coloneqq \nabla_{\mathcal{S}^{2}} \tilde{b}_{j} \times \mathbf{\tilde{N}}, \\
\end{aligned}
\label{eq:vecbasisfun}
\end{equation}
where $\mathbf{\tilde{N}}$ is the outward unit normal of $\mathcal{S}^{2}$.
See Fig.~\ref{fig:surfaces} for illustration.

\subsection{Scalar Spherical Harmonics} \label{sec:sphericalharmonics}

Let us consider the space of homogeneous harmonic polynomials in $\R^{3}$ which are of degree $n \in \N_{0}$.
We restrict their domain to the sphere $\mathcal{S}^{2}$ and denote this space by $\mathrm{Harm}_{n}$.
Then, by Thm.~5.6 in \cite[Sec.~5.1]{Mic13}, we have $\mathrm{dim}(\mathrm{Harm}_{n}) = 2n + 1$.

For $n \in \N_{0}$, an element $\tilde{Y}_{n} \in \mathrm{Harm}_{n}$ is an infinitely often differentiable eigenfunction of the Laplace-Beltrami operator $\Delta_{\mathcal{S}^{2}}$, as defined in \eqref{eq:laplacebeltrami}, and is referred to as a \emph{(scalar) spherical harmonic}.
Its corresponding eigenvalue is $\lambda_{n} = n(n + 1)$, see Lemma~5.8 in \cite{Mic13} for a proof.
Moreover, it holds that
\begin{equation}
	\langle \tilde{Y}_{n, j}, \tilde{Y}_{m, k} \rangle_{L^{2}(\mathcal{S}^{2})} = \delta_{mn} \delta_{jk},
\label{eq:spharmorthonormal}
\end{equation}
where $\langle \tilde{f}, \tilde{g} \rangle_{L^{2}(\mathcal{S}^{2})} \coloneqq \int_{\mathcal{S}^{2}} \tilde{f} \tilde{g} \; d\mathcal{S}^{2}$, cf. Thm.~5.9 in \cite{Mic13}.

The set $\{ \tilde{Y}_{n, j}: n \in \N_{0}, j = 1, \dots, 2n + 1 \}$ is a complete orthonormal system in $L^{2}(\mathcal{S}^{2})$ with respect to $\langle \cdot, \cdot \rangle_{L^{2}(\mathcal{S}^{2})}$.
As a consequence, every function $\tilde{f} \in L^{2}(\mathcal{S}^{2})$ can be uniquely expanded in its Fourier series representation as
\begin{equation*}
	\tilde{f} = \sum_{n = 0}^{\infty} \sum_{j = 1}^{2n + 1} \langle \tilde{f}, \tilde{Y}_{n, j} \rangle_{L^{2}(\mathcal{S}^{2})} \tilde{Y}_{n,j}.
\end{equation*}
See Thm.~5.25 in \cite{Mic13} for the details.
In this article, we will assume that $\tilde{Y}_{n, j} \in \mathrm{Harm}_{n}$ denote \emph{fully normalised spherical harmonics}, see \cite[Sec.~5.2]{Mic13} for their construction.
By Parseval's identity, we furthermore have
\begin{equation*}
	\norm{\tilde{f}}_{L^{2}(\mathcal{S}^{2})}^{2} = \sum_{n, j} \langle \tilde{f}, \tilde{Y}_{n, j} \rangle_{L^{2}(\mathcal{S}^{2})}^{2}.
\end{equation*}
Again, see Thm.~5.25 in \cite{Mic13}.

We define the Sobolev space $H^{r}(\mathcal{S}^{2})$ for arbitrary $r \in \R$ by means of the completion of all $C^{\infty}(\mathcal{S}^{2})$ functions with respect to the norm
\begin{equation*}
	\norm{\tilde{f}}_{H^{r}(\mathcal{S}^{2})}^{2} \coloneqq \norm{(\Delta_{\mathcal{S}^{2}} + 1)^{r/2} \tilde{f}}_{L^{2}(\mathcal{S}^{2})}^{2} = \sum_{n, j} (\lambda_{n} + 1)^{r} \langle \tilde{f}, \tilde{Y}_{n, j} \rangle_{L^{2}(\mathcal{S}^{2})}^{2}.
\end{equation*}
For $r \in \R$, we define the $H^{r}(\mathcal{S}^{2})$ seminorm of order $r$ by
\begin{equation}
	\abs{\tilde{f}}_{H^{r}(\mathcal{S}^{2})}^{2} \coloneqq \norm{\Delta_{\mathcal{S}^{2}}^{r/2} \tilde{f}}_{L^{2}(\mathcal{S}^{2})}^{2} = \sum_{n, j} \lambda_{n}^{r} \langle \tilde{f}, \tilde{Y}_{n, j} \rangle_{L^{2}(\mathcal{S}^{2})}^{2}.
\label{eq:sobolevseminormsphere}
\end{equation}
\section{Problem Formulation} \label{sec:model}

Let us consider an evolving sphere-like surface
\begin{equation}
	\mathcal{M} \coloneqq \bigcup_{t \in I} \bigl( \{ t \} \times \mathcal{M}_{t} \bigr) \subset \R^{4}
\label{eq:evolvsurf}
\end{equation}
which is specified in terms of a parametrisation $\y: I \times \Omega \to \R^{3}$ as in \eqref{eq:param}.
Every choice of $\y$ gives rise to a surface velocity
\begin{equation}
	\mathbf{\hat{V}}(t, x) = \partial_{t} \y(t, \xi) \in \R^{3},
\label{eq:surfvel}%
\end{equation}
where $\xi = \y^{-1}(t, x)$.
We stress that the velocity $\mathbf{\hat{V}}$ depends on the chosen parametrisation $\y$ of which, in general, infinitely many exist.
However, its (scalar) normal component, given by
\begin{equation*}
	V = \mathbf{\hat{V}} \cdot \mathbf{\hat{N}},
\end{equation*}
is intrinsic and thus independent of the choice of $\y$, see e.g.~\cite[Prop.~1]{KirLanSch15}.
As a consequence, \eqref{eq:surfvel} can be represented as
\begin{equation}
	\mathbf{\hat{V}} = V\mathbf{\hat{N}} + \mathbf{\hat{v}},
\label{eq:surfveldecomp}
\end{equation}
where $V\mathbf{\hat{N}}$ is the normal velocity and $\mathbf{\hat{v}}$ is a vector field tangent to $\mathcal{M}_{t}$, $t \in I$.

In the following, we consider smooth trajectories of moving particles (or cells) which always stay on the evolving surface.
To this end, we assume the existence of a Lagrangian specification
\begin{equation}
	\gamma(\cdot, x): t \mapsto \gamma(t, x) \in \mathcal{M}_{t}, \quad \gamma(0, \cdot) = \Id
\label{eq:trajectory}
\end{equation}
of the path of a particle which starts at $x \in \mathcal{M}_{0}$ and always stays on the surface.
Expressing \eqref{eq:trajectory} with the help of a coordinate representation $\beta: I \times \Omega \to \Omega$ requires that
\begin{equation}
	\gamma(t, \y(0, \xi)) = \y(t, \beta(t, \xi)), \quad \beta(0, \cdot) = \Id
\label{eq:trajectoryparam}
\end{equation}
holds for all $(t \times \xi) \in I \times \Omega$.
As a consequence of~\eqref{eq:trajectoryparam} and with the help of \eqref{eq:surfvel} we find that
\begin{equation}
\begin{aligned}
	\partial_{t} \gamma & = \partial_{t} \y + \partial_{t} \beta^{i} \partial_{i} \y, \\
	& = \mathbf{\hat{V}} + \mathbf{\hat{w}},
\end{aligned}
\label{eq:velocity}
\end{equation}
where $\mathbf{\hat{w}} = \partial_{t} \beta^{i} \partial_{i} \y$ is a purely tangential velocity.
Therefore, the velocity of a particle moving along \eqref{eq:trajectory} can be decomposed into the surface velocity $\mathbf{\hat{V}}$, which is prescribed by the chosen parametrisation $\y$, and a tangential part $\mathbf{\hat{w}}$ relative to it.
See Fig.~\ref{fig:sketch} for a sketch.

As a consequence of \eqref{eq:trajectory} and \eqref{eq:trajectoryparam} we infer that the normal part of the velocity of a particle following $\gamma$ equals the normal velocity of the surface movement.
In other words,
\begin{align*}
	\partial_{t} \gamma \cdot \mathbf{\hat{N}} & = (\mathbf{\hat{V}} + \mathbf{\hat{w}}) \cdot \mathbf{\hat{N}} \\
	& = \mathbf{\hat{V}} \cdot \mathbf{\hat{N}} \\
	& = V.
\end{align*}

Suppose now that the evolving surface \eqref{eq:evolvsurf} is embedded in a fluid which moves with a velocity $\mathbf{U}(t, x) \in \R^{3}$, $x \in \R^{3}$.
For $t \in I$ and $x \in \mathcal{M}_{t}$, we denote the restriction of $\mathbf{U}(t, x)$ to the surface $\mathcal{M}_{t}$ by $\mathbf{\hat{U}}(t, x)$.
We stress that this fluid velocity is in general different from the surface velocity $\mathbf{\hat{V}}$, defined in~\eqref{eq:surfvel}.

In the following we assume that a particle of interest following~\eqref{eq:trajectory} convects with this fluid.
In other words, for $t \in I$ and $x \in \mathcal{M}_{t}$ we require that
\begin{equation}
	\mathbf{\hat{U}}(t, x) = \partial_{t} \gamma(t, \gamma^{-1}(t, x)).
\label{eq:trajectoryvel}
\end{equation}
From \eqref{eq:velocity} and \eqref{eq:surfveldecomp} we find that
\begin{equation}
\begin{aligned}
	\mathbf{\hat{U}} & = \mathbf{\hat{V}} + \mathbf{\hat{w}} \\
	& = V \mathbf{\hat{N}} + \mathbf{\hat{v}} + \mathbf{\hat{w}}.
\end{aligned}
\label{eq:decomp}
\end{equation}
Therefore, the surface \eqref{eq:evolvsurf} must evolve with (scalar) normal velocity $V = \mathbf{\hat{U}} \cdot \mathbf{\hat{N}}$.
Since the fluid velocity $\mathbf{\hat{U}}$ can uniquely be decomposed into a normal and a tangential part, we conclude that the latter is given by
\begin{equation}
	\mathbf{\hat{u}} = \mathbf{\hat{v}} + \mathbf{\hat{w}}.
\label{eq:tangentialmotion}
\end{equation}

The primary goal of this article is to estimate the motion of cells as they move along trajectories \eqref{eq:trajectory} through Euclidean 3-space.
The main assumption is that they form a surface structure which is deforming over time and can be estimated from image data $\hat{f}$.
Hence, we focus on estimating $\mathbf{\hat{U}}$ rather than $\mathbf{U}$ and utilise the fact that the unknown can be decomposed as in \eqref{eq:decomp}.

In the following we discuss two conceptually different ways of estimating the tangential part of the particle motion.
One is based on conservation of the data $\hat{f}$ along paths \eqref{eq:trajectory} and leads to a generalised optical flow equation.
Given $\hat{f}$ and a surface velocity $\mathbf{\hat{V}}$, one tries to compute a tangential vector field $\mathbf{\hat{w}}$ relative to it.
This precise approach has been pursued already in \cite{KirLanSch13, KirLanSch15, LanSch17}.

The other idea is based on conservation of mass and leads to a suitable generalisation of the continuity equation to evolving surfaces.
Given $\hat{f}$ and only the normal component $V$ of the surface velocity, one directly tries to infer the entire tangential part $\mathbf{\hat{u}}$ of the particle motion.

The main differences are as follows.
First, they differ in the assumptions imposed on $\hat{f}$.
One assumes conservation of brightness whereas the other assumes conservation of mass.
Second, in the former the unknown is $\mathbf{\hat{w}}$, whereas in the latter the unknown is $\mathbf{\hat{u}}$.
Third, as we employ a variational approach, they differ in their regularity assumptions.
The first approach desires regularity of $\mathbf{\hat{w}}$, which depends on the tangential part of the imposed surface velocity $\mathbf{\hat{V}}$, whereas the second enforces regularity of the tangential part $\mathbf{\hat{u}}$ of the desired motion.

\subsection{Conservation of Brightness} \label{sec:brightness}

Let us be given a function $\hat{f}$ such that, for time $t \in I$,
\begin{equation*}
	\hat{f}(t, \cdot): \mathcal{M}_{t} \to \R
\end{equation*}
is an image on the surface $\mathcal{M}_{t}$.
In this section we assume that, along a smooth trajectory \eqref{eq:trajectory}, this data $\hat{f}$ satisfies
\begin{equation}
	\hat{f}(t, \gamma(t, x)) = \hat{f}(0, x)
\label{eq:bca}
\end{equation}
for all $t \in I$ and all $x \in \mathcal{M}_{0}$.
Typically, this constraint is termed \emph{brightness constancy assumption} and is the basis for many motion estimation methods.

In order to linearise \eqref{eq:bca} by differentiation with respect to time, one may consider temporal derivatives along trajectories, see \cite{KirLanSch13, KirLanSch15}.
To this end, we define the time derivative of $\hat{f}$ along a trajectory $\psi: t \mapsto \psi(t) \in \mathcal{M}_{t}$ at $x_{0} = \psi(t_{0})$ as
\begin{equation}
	d_{t}^{\partial_{t} \psi} \hat{f}(t_{0}, x_{0}) \coloneqq \frac{d}{dt} \hat{f}(t, \psi(t)) \bigg\vert_{t = t_{0}}.
\label{eq:timederiv}
\end{equation}
In further consequence, the time derivative of $\hat{f}$ at $x_{0} = \y(t_{0}, \xi)$ along the parametrisation $\y(\cdot, \xi)$ is defined analogously as
\begin{equation}
	d_{t}^{\mathbf{\hat{V}}} \hat{f}(t_{0}, x_{0}) \coloneqq \frac{d}{dt} \hat{f}(t, \y(t, \xi)) \bigg\vert_{t = t_{0}}.
\label{eq:paramtimederiv}
\end{equation}
For a trajectory $\psi_{\mathbf{\hat{N}}}$ that passes through $x_{0} \in \mathcal{M}_{t_{0}}$ at time $t_{0}$ and for which $\partial_{t} \psi_{\mathbf{\hat{N}}}(t_{0})$ is orthogonal to $T_{x_{0}}\mathcal{M}_{t_{0}}$, the so-called \emph{normal time derivative} of $\hat{f}$ is defined as
\begin{equation}
	d_{t}^{\mathbf{\hat{N}}} \hat{f}(t_{0}, x_{0}) \coloneqq \frac{d}{dt} \hat{f}(t, \psi_{\mathbf{\hat{N}}}(t)) \bigg\vert_{t = t_{0}}.
\label{eq:normaltimederiv}
\end{equation}
The relation between \eqref{eq:timederiv} and \eqref{eq:normaltimederiv} is given by
\begin{equation}
	d_{t}^{\partial_{t} \psi} \hat{f} = d_{t}^{\mathbf{\hat{N}}} \hat{f} + \nabla_{\mathcal{M}} \hat{f} \cdot \partial_{t} \psi.
\label{eq:timederivrelation}
\end{equation}
See \cite[Sec.~3.3]{CerFriGur05} for the details.
Figure~\ref{fig:sketch} shows a sketch of the different trajectories introduced above and their velocities.

\begin{figure}[t]
	\begin{center}
	\begin{tikzpicture}[scale=0.8]
		\draw [thick, gray] (-3,0) to [out=30,in=150] (3,0);
		\node [right] at (3,0) {$\mathcal{M}_{t_0}$};		
		\draw [thick, gray] (-3,1) to [out=40,in=150] (4,1.8);
		\node [below] at (4,1.8) {$\mathcal{M}_{t_0+\Delta t}$};
		\draw [-stealth', gray, thick] (0,.9) -- (.91*1.1,.91*2.6);
		\node [left, gray] at (.91,.91*2.5) {$\mathbf{\hat{V}}$};
		\draw [-stealth', gray, thick] (0,.9) -- (0,2.4);
		\draw [-stealth', gray, thick] (0,.9) -- (.94*2.5,.94*2.4);
		\node [below, gray] at (.98*2.5,.94*2.4) {$\mathbf{\hat{U}}$};
		\draw [-stealth', gray, thick] (0,.9) -- (.94*2.5-.91*1.1, 0.9);
		\node [right, gray] at (0.94*2.5-0.91*1.1, 0.9) {$\mathbf{\hat{w}}$};
		\draw [-stealth', thick, dotted] (0,0) to [out=90,in=225] (0,.9) to [out=45,in=270] (1.3,1.7) to [out=90,in=225] (2.5,2.4) to [out=45,in=225] (1.25*2.5,1.25*2.4);
		\node [right] at (1.3*2.5,1.3*2.4) {$\gamma(\cdot, x)$};
		\draw [-stealth', thick] (0.5,0) to [out=135,in=270] (0,.9) to [out=90,in=300] (-.5,1.25*2.4);
		\node [above] at (-.5,1.25*2.4) {$\psi_{\mathbf{\hat{N}}}$};
		\draw [-stealth', thick, dashed] (-0.5,0) to [out=45,in=225] (0,.9) to [out=45,in=270] (1.5,3.1);
		\node [above] at (1.5,3.1) {$\y(\cdot, \xi)$};
		\node at (0.4, 0.6) {$x_{0}$};
	\end{tikzpicture}
	\end{center}
	\caption{Sketch of various trajectories following the evolving surface. The corresponding velocities are depicted in grey. The velocity $\mathbf{\hat{U}}$ of a cell following $\gamma$ is composed of the surface velocity $\mathbf{\hat{V}}$ and a tangential velocity $\mathbf{\hat{w}}$.}
	\label{fig:sketch}
\end{figure}
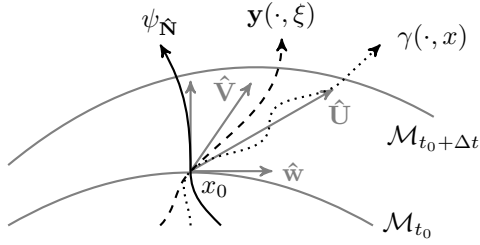

Recall that by assumption \eqref{eq:trajectoryvel} we have $\mathbf{\hat{U}} = \partial_{t} \gamma$.
With the help of definition \eqref{eq:timederiv} and relation \eqref{eq:timederivrelation} we can immediately recast assumption \eqref{eq:bca} and demand that along a trajectory $\gamma$, as defined in \eqref{eq:trajectory}, we must have
\begin{equation}
	d_{t}^{\mathbf{\hat{U}}} \hat{f} = d_{t}^{\mathbf{\hat{N}}} \hat{f} + \nabla_{\mathcal{M}} \hat{f} \cdot \mathbf{\hat{U}} \overset{\mathclap{!}}{=} 0.
\label{eq:gofe}
\end{equation}
However, this so-called \emph{generalised optical flow equation} is inconvenient from a numerical perspective, as $d_{t}^{\mathbf{\hat{N}}} \hat{f}$ typically is unknown or hard to estimate from real data.
As a remedy, they propose in \cite[Lemma~2]{KirLanSch15} to use
\begin{equation*}
\begin{aligned}
	d_{t}^{\mathbf{\hat{N}}} \hat{f} + \nabla_{\mathcal{M}} \hat{f} \cdot \mathbf{\hat{U}} \; & \overset{\mathclap{\eqref{eq:decomp}}}{=} d_{t}^{\mathbf{\hat{N}}} \hat{f} + \nabla_{\mathcal{M}} \hat{f} \cdot (\mathbf{\hat{V}} + \mathbf{\hat{w}}) \\
	& \overset{\mathclap{\eqref{eq:timederivrelation}}}{=} d_{t}^{\mathbf{\hat{V}}} \hat{f} - \nabla_{\mathcal{M}} \hat{f} \cdot \mathbf{\hat{V}} + \nabla_{\mathcal{M}} \hat{f} \cdot (\mathbf{\hat{V}} + \mathbf{\hat{w}}) \\
	& = d_{t}^{\mathbf{\hat{V}}} \hat{f} + \nabla_{\mathcal{M}} \hat{f} \cdot \mathbf{\hat{w}},
\end{aligned}
\end{equation*}
which is a parametrised version of \eqref{eq:gofe} and is referred to as \emph{parametrised optical flow equation}.
We highlight that the unknown $\mathbf{\hat{w}}$ depends exclusively on the imposed surface velocity $\mathbf{\hat{V}}$.

Computing the optical flow $\mathbf{\hat{w}}$ from data $\hat{f}$ constitutes an ill-posed inverse problem as the above equation is underdetermined and a unique solution is not guaranteed.
As a remedy, we minimise a Tikhonov-type functional consisting of a data term and a spatially varying regularisation term.

In what follows, we keep $t \in I$ arbitrary but fixed and seek a minimiser to the functional $\mathcal{E}: H^{1}(\mathcal{M}_{t}, T\mathcal{M}_{t}) \to [0, +\infty]$,
\begin{equation}
	\mathcal{E}(\mathbf{\hat{w}}) \coloneqq \norm{d_{t}^{\mathbf{\hat{V}}} \hat{f} + \nabla_{\mathcal{M}} \hat{f} \cdot \mathbf{\hat{w}}}_{L^{2}(\mathcal{M}_{t})}^{2} + \mathcal{R}(\mathbf{\hat{w}}),
\label{eq:offunctional}
\end{equation}
where, given a measurable function $s(t, \cdot): \mathcal{M}_{t} \to \{0, 1\}$, the regularisation functional $\mathcal{R}(\mathbf{\hat{w}})$ is defined as
\begin{equation}
	\mathcal{R}(\mathbf{\hat{w}}) \coloneqq \alpha_{0} \int_{\mathcal{M}_{t}} s \norm{\nabla \mathbf{\hat{w}}}_{2}^{2} \; d\mathcal{M}_{t} + \alpha_{1} \int_{\mathcal{M}_{t}} (1-s) \norm{\mathbf{\hat{w}}}^{2} \; d\mathcal{M}_{t},
\label{eq:regfunctional}
\end{equation}
and $\alpha_{0}, \alpha_{1} > 0$ are regularisation parameters.
Here, the function $s$ incorporates a-priori information about the support of the solution.
The idea is that minimisation of \eqref{eq:offunctional} with \eqref{eq:regfunctional} as regularisation functional favours tangent vector fields of certain regularity in areas where data is present but, on the other hand, prevents potentially undesired fill-in effects of quadratic regularisation in regions with no data.
In practice one may choose $s$ to be e.g. a segmentation of the fluorescently labelled cells or, due to the nature of the fluorescence microscopy data, one may as well choose $s \coloneqq \hat{f}$ with $\hat{f}(t, \cdot): \mathcal{M}_{t} \to [0, 1]$.

In this article, however, we restrict ourselves to functions $s(t, \cdot): \mathcal{M}_{t} \to (0, 1)$ as, by the equivalence of norms (see Sec.~\ref{sec:background:sobolevspaces}), coercivity of $\mathcal{E}$ with respect to $H^{1}(\mathcal{M}_{t}, T\mathcal{M}_{t})$ and thus well-posedness of the problem is guaranteed.
We refer to \cite{BauGraKir15} for further details.
For the actual choice of $s$ see Sec.~\ref{sec:experiments}.

Moreover, let us emphasise that \eqref{eq:offunctional} is a generalisation of the variational formulation used in \cite{LanSch17}, where only the Sobolev (semi-)norm \eqref{eq:h1norm} was used as regularisation functional.
It is immediately recovered by choosing $s \equiv 1$.

\subsection{Conservation of Mass}

Let us be given a time-evolving surface \eqref{eq:evolvsurf} which is migrating through a fluid defined in the ambient space.
We stress that, in general, this surface is \emph{non-material}.
In other words, the surface velocity $\mathbf{\hat{V}}$ induced by a chosen parametrisation of $\mathcal{M}$ is different from the fluid velocity $\mathbf{\hat{U}}$.

Furthermore, let us denote by
\begin{equation*}
	\hat{f}(t, \cdot): \mathcal{M}_{t} \to \R,
\end{equation*}
the density of the fluid restricted to the surface $\mathcal{M}_{t}$.
With the goal of estimating the fluid motion, we assume that this data $\hat{f}$ satisfies mass preservation.

In order to derive a suitable conservation law, let us consider an arbitrary evolving subsurface $\Gamma_{t} \subseteq \mathcal{M}_{t}$ of this surface.
For the sake of simplicity we will omit the index and write $\Gamma$, respectively $\partial \Gamma$ for the subsurface and its boundary.

The boundary curve $\partial \Gamma$ is oriented by its exterior unit normal field $\boldsymbol{\hat{\nu}}$.
Recall that $\boldsymbol{\hat{\nu}}$ is normal to $\partial \Gamma$ and tangent to $\mathcal{M}_{t}$.
We denote by $\mathbf{\hat{V}}_{\partial \Gamma} \in \R^{3}$ the velocity of the curve $\partial \Gamma$ as it moves through the embedding space.
Its intrinsic component, which is independent of the parametrisation of the curve $\partial \Gamma$, is denoted by
\begin{equation*}
	V_{\partial \Gamma} \coloneqq \mathbf{\hat{V}}_{\partial \Gamma} \cdot \boldsymbol{\hat{\nu}}.
\end{equation*}
Since by assumption $\Gamma \subseteq \mathcal{M}_{t}$, we deduce that
\begin{equation*}
	\mathbf{\hat{V}}_{\partial \Gamma} \cdot \mathbf{\hat{N}} = V.
\end{equation*}
In other words, $\Gamma$ and $\mathcal{M}_{t}$ evolve with equal normal velocities.
Furthermore, the \emph{normal migrational velocity} $V_{\partial \Gamma}^{\text{mig}}$ of the curve $\partial \Gamma$, as it travels through the fluid, is defined by
\begin{equation*}
	V_{\partial \Gamma}^{\text{mig}} \coloneqq (\mathbf{\hat{V}}_{\partial \Gamma} - \mathbf{\hat{U}}) \cdot \boldsymbol{\hat{\nu}}.
\end{equation*}

Given a fluid density $\hat{f}$ and an arbitrary evolving subsurface $\Gamma \subseteq \mathcal{M}_{t}$, the transport relation
\begin{equation}
	\frac{d}{dt} \int_{\Gamma} \hat{f} \; d\Gamma = \int_{\Gamma} \left( d_{t}^{\mathbf{\hat{N}}} \hat{f} + \nabla_{\mathcal{M}} \cdot (\hat{f} \mathbf{\hat{u}}) - \hat{f}KV \right) \; d\Gamma + \int_{\partial \Gamma} \hat{f} V_{\partial \Gamma}^{\text{mig}} \; d\Gamma
\label{eq:mbtransport}
\end{equation}
holds.
We refer to \cite[Sec.~4.2]{CerFriGur05} for the details.
Here, $\mathbf{\hat{u}}$ is the tangent part \eqref{eq:tangentialmotion} of the fluid velocity, $\nabla_{\mathcal{M}} \cdot (\hat{f} \mathbf{\hat{u}})$ denotes the surface divergence of $\hat{f} \mathbf{\hat{u}}$, and $K$ is the total curvature, see \eqref{eq:surfdiv} and \eqref{eq:totalcurv}, respectively.

Recall that at the beginning of this section we have assumed that the surface $\mathcal{M}_{t}$ evolves with (scalar) normal velocity $V = \mathbf{\hat{U}} \cdot \mathbf{\hat{N}}$, see \eqref{eq:trajectoryvel}.
In addition, let us suppose that $\Gamma$ is \emph{material}, meaning that it convects with the fluid.
In other words, it holds that
\begin{equation*}
	\mathbf{\hat{V}}_{\partial \Gamma} = \mathbf{\hat{U}}.
\end{equation*}
As a consequence, we have $V_{\partial\Gamma}^{\text{mig}} = 0$ and the transport relation \eqref{eq:mbtransport} simplifies to
\begin{equation*}
	\frac{d}{dt} \int_{\Gamma} \hat{f} \; d\Gamma = \int_{\Gamma} \left( d_{t}^{\mathbf{\hat{N}}} \hat{f} + \nabla_{\mathcal{M}} \cdot (\hat{f} \mathbf{\hat{u}}) - \hat{f}KV \right) \; d\Gamma.
\end{equation*}

Since $\Gamma$ is material, conservation of mass requires that
\begin{equation*}
	\frac{d}{dt} \int_{\Gamma} \hat{f} \; d\Gamma \overset{\mathclap{!}}{=} 0
\end{equation*}
and we obtain the relation
\begin{equation*}
	\int_{\Gamma} \left( d_{t}^{\mathbf{\hat{N}}} \hat{f} + \nabla_{\mathcal{M}} \cdot (\hat{f} \mathbf{\hat{u}}) - \hat{f}KV \right) \; d\Gamma = 0.
\end{equation*}
Since $\Gamma$ was arbitrary, this leads to the point-wise conservation law
\begin{equation}
	d_{t}^{\mathbf{\hat{N}}} \hat{f} + \nabla_{\mathcal{M}} \cdot (\hat{f} \mathbf{\hat{u}}) - \hat{f}KV = 0,
\label{eq:gce}
\end{equation}
which, following the terminology from before, resembles a \emph{generalised continuity equation}.
As for the generalised optical flow equation, we utilise relation \eqref{eq:timederivrelation} with time derivative \eqref{eq:paramtimederiv} and obtain a parametrised mass preservation constraint
\begin{equation}
	d_{t}^{\mathbf{\hat{V}}} \hat{f} + \nabla_{\mathcal{M}} \cdot (\hat{f} \mathbf{\hat{u}}) - \hat{f}KV - \nabla_{\mathcal{M}} \hat{f} \cdot \mathbf{\hat{v}} = 0,
\label{eq:pce}
\end{equation}
where $\mathbf{\hat{v}}$ is the tangent part of the surface velocity \eqref{eq:surfveldecomp}.
We refer to it as \emph{parametrised continuity equation}.

Let us mention that, with the help of \eqref{eq:surfveldecomp} and \eqref{eq:gofe}, one can alternatively rewrite \eqref{eq:gce} and solve for $\mathbf{\hat{w}}$ in
\begin{equation*}
	d_{t}^{\mathbf{\hat{V}}} \hat{f} + \nabla_{\mathcal{M}} \cdot (\hat{f} \mathbf{\hat{w}}) - \hat{f}KV - \hat{f} \nabla_{\mathcal{M}} \cdot \mathbf{\hat{v}} = 0.
\end{equation*}
However, for the reasons elaborated in Sec.~\ref{sec:model}, we consider solving \eqref{eq:pce} in a variational formulation.

Again, let $t \in I$ be fixed.
In view of the ill-posedness of \eqref{eq:pce}, we seek a minimiser to the functional $\mathcal{F}: H^{1}(\mathcal{M}_{t}, T\mathcal{M}_{t}) \to [0, +\infty]$,
\begin{equation}
	\mathcal{F}(\mathbf{\hat{u}}) \coloneqq \norm{d_{t}^{\mathbf{\hat{N}}} \hat{f} + \nabla_{\mathcal{M}} \cdot (\hat{f} \mathbf{\hat{u}}) - \hat{f} K V}_{L^{2}(\mathcal{M}_{t})}^{2} + \mathcal{R}(\mathbf{\hat{u}}) + \mathcal{S}(\mathbf{\hat{u}}),
\label{eq:cmfunctional}
\end{equation}
where $\mathcal{R}(\mathbf{\hat{u}})$ is defined as in \eqref{eq:regfunctional} and
\begin{equation}
	\mathcal{S}(\mathbf{\hat{u}}) \coloneqq \alpha_{2} \int_{\mathcal{M}_{t}} (1 - s) \bigl( \nabla_{\mathcal{M}} \cdot \mathbf{\hat{u}} \bigr)^{2} \; d\mathcal{M}_{t}.
\label{eq:regfunctional2}
\end{equation}
Here, $\alpha_{2} > 0$ is an additional regularisation parameter.
The reason for this additional regularisation term in contrast to \eqref{eq:offunctional} is to control oscillations in the velocity field, which may arise from the data term in the presence of noise.
For the concrete choice of $s$ we again refer to Sec.~\ref{sec:experiments}.
\section{Numerical Solution} \label{sec:numerics}

In the following we consider $t \in I$ arbitrary but fixed.
Let us be given a set $\{x_{j} \in \mathcal{S}^{2}\}_{j=1, \dots, N}$ of pairwise distinct points on the 2-sphere.
With each of its elements $x_{j}$ we associate the $x_{j}$-zonal function $\tilde{b}_{j}$, see \eqref{eq:zonalfunc}.
According to definition \eqref{eq:vecbasisfun}, we immediately obtain the set
\begin{equation}
	\left\{ \mathbf{\tilde{y}}_{j}^{(i)}: j = 1, \dots, N, i = 1, 2 \right\}
\label{eq:vecbasisfunset}
\end{equation}
of tangent vector fields on $\mathcal{S}^{2}$.

We approximate the solutions to the problems
\begin{equation*}
	\min_{\mathbf{\hat{w}} \in H^{1}(\mathcal{M}_{t}, T\mathcal{M}_{t})} \mathcal{E}(\mathbf{\hat{w}}) \quad \text{and} \quad \min_{\mathbf{\hat{u}} \in H^{1}(\mathcal{M}_{t}, T\mathcal{M}_{t})} \mathcal{F}(\mathbf{\hat{u}})
\end{equation*}
in a finite-dimensional subspace $\mathcal{U}$, where $\mathcal{E}$ and $\mathcal{F}$ are defined as in \eqref{eq:offunctional} and \eqref{eq:cmfunctional}, respectively.
We define this space of tangent vector fields on $\mathcal{M}_{t}$ as
\begin{equation}
	\mathcal{U} \coloneqq \mathrm{span} \left\{ \mathbf{\hat{y}}_{j}^{(i)}: j = 1, \dots, N, i = 1, 2 \right\}.
\label{eq:vecbasisfunspace}
\end{equation}
Here, $\mathbf{\hat{y}}_{j}^{(i)} = D\tilde{\phi}(t, \cdot) \mathbf{\tilde{y}}_{j}^{(i)}$ is the pushforward of an element $\mathbf{\tilde{y}}_{j}^{(i)}$ contained in the set \eqref{eq:vecbasisfunset} by means of the differential $D\tilde{\phi}$, see \eqref{eq:dphi} for its definition.
For notational convenience we relabel the elements of $\mathcal{U}$ with the help of an index set $J_{\mathcal{U}} \subset \N$ and use a single index letter $p \in J_{\mathcal{U}}$.

\subsection{Conservation of Brightness}

We expand the sought tangent vector field as
\begin{equation}
	\mathbf{\hat{w}} = \sum_{p \in J_{\mathcal{U}}} w_{p} \mathbf{\hat{y}}_{p},
\label{eq:ofansatz}
\end{equation}
where $w_{p} \in \R$, $p \in J_{\mathcal{U}}$, are the unknown coefficients.
By plugging ansatz \eqref{eq:ofansatz} into functional \eqref{eq:offunctional}, we obtain for the data term
\begin{equation*}
	\int_{\mathcal{M}_{t}} \Bigl( d_{t}^{\mathbf{\hat{V}}} \hat{f} + \sum_{p \in J_{\mathcal{U}}} w_{p} \bigl( \nabla_{\mathcal{M}} \hat{f} \cdot \mathbf{\hat{y}}_{p} \bigr) \Bigr)^{2} \; d\mathcal{M}_{t}.
\end{equation*}

Concerning the regularisation functional $\mathcal{R}(\mathbf{\hat{w}})$, as defined in \eqref{eq:regfunctional}, we first observe that the coefficients $D_{i}w^{k}$, defined in \eqref{eq:covderivcoeff}, are linear.
For $\mathbf{\hat{y}}_{p} = y_{p}^{k} \partial_{k} \y$ we have, by definition \eqref{eq:covderivcoeff},
\begin{align*}
	D_{i}w^{k} & = \partial_{i} \left( \sum_{p \in J_{\mathcal{U}}} w_{p} y_{p}^{k}  \right) + \left( \sum_{p \in J_{\mathcal{U}}} w_{p} y_{p}^{m} \right) \Gamma_{im}^{k} \\
	& = \sum_{p \in J_{\mathcal{U}}} w_{p} \left( \partial_{i} y_{p}^{k} + y_{p}^{m} \Gamma_{im}^{k} \right) \\
	& = \sum_{p \in J_{\mathcal{U}}} w_{p} D_{i}y_{p}^{k}.
\end{align*}
With the help of Lemma~\ref{lem:covderiv} we then find that
\begin{equation*}
	\norm{\nabla \sum_{p \in J_{\mathcal{U}}} w_{p} \mathbf{\hat{y}}_{p}}_{2}^{2} = \sum_{p, q \in J_{\mathcal{U}}} w_{p} w_{q} g_{k\ell} g^{ij} D_{i}y_{p}^{k} D_{j}y_{q}^{\ell}
\end{equation*}
and, moreover, for the second term in \eqref{eq:regfunctional} we obtain
\begin{equation*}
	\norm{\sum_{p \in J_{\mathcal{U}}} w_{p} \mathbf{\hat{y}}_{p}}^{2} = \sum_{p, q \in J_{\mathcal{U}}} w_{p} w_{q} \bigl( \mathbf{\hat{y}}_{p} \cdot \mathbf{\hat{y}}_{q} \bigr).
\end{equation*}

The optimality conditions for $\mathcal{E}(\mathbf{\hat{w}})$ are obtained by  taking $\partial \mathcal{E}/\partial w_{p} = 0$ for all $p \in J_{\mathcal{U}}$ and in matrix-vector form read
\begin{equation}
	(A + \alpha_{0} C + \alpha_{1} D) w = b,
\label{eq:linsystemof}
\end{equation}
where $w = (w_{1}, \dots, w_{\abs{J_{\mathcal{U}}}})^{\top} \in \R^{\abs{J_{\mathcal{U}}}}$ denotes the vector of unknowns.
The entries of the matrix $A = (a_{pq})$ corresponding to the data term are given by
\begin{equation*}
	a_{pq} = \int_{\mathcal{M}_{t}} \bigl( \nabla_{\mathcal{M}} \hat{f} \cdot \mathbf{\hat{y}}_{p} \bigr) \bigl( \nabla_{\mathcal{M}} \hat{f} \cdot \mathbf{\hat{y}}_{q} \bigr) \; d\mathcal{M}_{t},
\end{equation*}
whereas the entries of the matrices $C = (c_{pq})$ and $D = (d_{pq})$ corresponding to the regularisation terms are given by
\begin{equation*}
	c_{pq} = \int_{\mathcal{M}_{t}} s g_{k\ell} g^{ij} D_{i}y_{p}^{k} D_{j}y_{q}^{\ell} \; d\mathcal{M}_{t}
\end{equation*}
and
\begin{equation*}
	d_{pq} = \int_{\mathcal{M}_{t}} (1 - s) \bigl( \mathbf{\hat{y}}_{p} \cdot \mathbf{\hat{y}}_{q} \bigr) \; d\mathcal{M}_{t},
\end{equation*}
respectively.
The entries of the vector $b = (b_{p})$ are
\begin{equation*}
	b_{p} = - \int_{\mathcal{M}_{t}} d_{t}^{\mathbf{\hat{V}}} \hat{f} \bigl( \nabla_{\mathcal{M}} \hat{f} \cdot \mathbf{\hat{y}}_{p} \bigr) \; d\mathcal{M}_{t}.
\end{equation*}

\subsection{Conservation of Mass}

Next, let us derive the optimality conditions for the functional $\mathcal{F}$, defined in \eqref{eq:cmfunctional}.
For numerical convenience we use \eqref{eq:pce} rather than \eqref{eq:gce} as they are equivalent.
Accordingly, we expand the sought tangent vector field as
\begin{equation*}
	\mathbf{\hat{u}} = \sum_{p \in J_{\mathcal{U}}} u_{p} \mathbf{\hat{y}}_{p},
\end{equation*}
where $u_{p} \in \R$, $p \in J_{\mathcal{U}}$, are the unknown coefficients.
For the data term we get
\begin{equation*}
	\int_{\mathcal{M}_{t}} \Bigl( d_{t}^{\mathbf{\hat{V}}} \hat{f} + \sum_{p \in J_{\mathcal{U}}} u_{p} \bigl( \nabla_{\mathcal{M}} \hat{f} \cdot \mathbf{\hat{y}}_{p} + \hat{f} \nabla_{\mathcal{M}} \cdot \mathbf{\hat{y}}_{p} \bigr) - \hat{f} K V - \nabla_{\mathcal{M}} \hat{f} \cdot \mathbf{\hat{v}} \Bigr)^{2} \; d\mathcal{M}_{t}.
\end{equation*}

Regarding the term $\mathcal{S}(\mathbf{\hat{u}})$ in the functional \eqref{eq:cmfunctional} we find that
\begin{equation*}
	\biggl( \nabla_{\mathcal{M}} \cdot \sum_{p \in J_{\mathcal{U}}} u_{p} \mathbf{\hat{y}}_{p} \biggr)^{2} = \sum_{p, q \in J_{\mathcal{U}}} u_{p} u_{q} \bigl( \nabla_{\mathcal{M}} \cdot \mathbf{\hat{y}}_{p} \bigr) \bigl( \nabla_{\mathcal{M}} \cdot \mathbf{\hat{y}}_{q} \bigr).
\end{equation*}
Analogously to before, by taking $\partial \mathcal{F}/\partial u_{p} = 0$ for all $p \in J_{\mathcal{U}}$ we obtain the optimality conditions in matrix-vector form
\begin{equation}
	(A + \alpha_{0} C + \alpha_{1} D + \alpha_{2} E) u = b,
\label{eq:linsystemcm}
\end{equation}
where the matrices $C$ and $E$ are as before.
The entries of the matrix $A = (a_{pq})$ corresponding to the data term are
\begin{equation*}
	a_{pq} = \int_{\mathcal{M}_{t}} \bigl( \nabla_{\mathcal{M}} \hat{f} \cdot \mathbf{\hat{y}}_{p} + \hat{f} \nabla_{\mathcal{M}} \cdot \mathbf{\hat{y}}_{p} \bigr) \bigl( \nabla_{\mathcal{M}} \hat{f} \cdot \mathbf{\hat{y}}_{q} + \hat{f} \nabla_{\mathcal{M}} \cdot \mathbf{\hat{y}}_{q} \bigr) \; d\mathcal{M}_{t}.
\end{equation*}
The entries of the matrix $E = (e_{pq})$ correspond to the regularisation term \eqref{eq:regfunctional2} and are given by
\begin{equation*}
	e_{pq} = \int_{\mathcal{M}_{t}} (1 - s) \bigl( \nabla_{\mathcal{M}} \cdot \mathbf{\hat{y}}_{p} \bigr) \bigl( \nabla_{\mathcal{M}} \cdot \mathbf{\hat{y}}_{q} \bigr) \; d\mathcal{M}_{t}.
\end{equation*}
Finally, the vector $b = (b_{p})$ now reads
\begin{equation*}
	b_{p} = - \int_{\mathcal{M}_{t}} \bigl( d_{t}^{\mathbf{\hat{V}}} \hat{f} - \hat{f} K V - \nabla_{\mathcal{M}} \hat{f} \cdot \mathbf{\hat{v}} \bigr) \bigl( \nabla_{\mathcal{M}} \hat{f} \cdot \mathbf{\hat{y}}_{p} \bigr) \; d\mathcal{M}_{t}.
\end{equation*}

\subsection{Surface Parametrisation} \label{sec:surfparam}

The main goal of this subsection is to estimate the time-evolving surface $\mathcal{M}$ together with a parametrisation $\y$ of the form \eqref{eq:param}.
We extend the idea of surface interpolation from scattered data in \cite{LanSch17} and seek a function $\tilde{\rho}: I \times \mathcal{S}^{2} \to (0, \infty)$ which is sufficiently regular in time and in space.

Given noisy data $\tilde{\rho}^{\delta}: I \times \mathcal{S}^{2} \to (0, \infty)$, we seek a minimiser to the energy
\begin{equation}
	\mathcal{G}(\tilde{\rho}) \coloneqq \int_{I} \Bigl( \norm{\tilde{\rho}(t, \cdot) - \tilde{\rho}^{\delta}(t, \cdot)}_{L^{2}(\mathcal{S}^{2})}^{2} + \beta_{0} \abs{\tilde{\rho}(t, \cdot)}_{H^{r}(\mathcal{S}^{2})}^{2} + \beta_{1} \norm{\partial_{t} \tilde{\rho}(t, \cdot)}_{L^{2}(\mathcal{S}^{2})}^{2} \Bigr) \; dt,
\label{eq:surffunctional}
\end{equation}
such that $\tilde{\rho} \in L^{2}(I; H^{r}(\mathcal{S}^{2}))$ and $\partial_{t} \tilde{\rho} \in L^{2}(I; L^{2}(\mathcal{S}^{2}))$.
We assume that $\tilde{\rho}^{\delta}(t, \cdot)$ is bounded for each $t \in I$.
Here, $\beta_{0}, \beta_{1} > 0$ are regularisation parameters balancing the terms, $r > 0$ is a sufficiently large real number, cf. \eqref{eq:sobolevseminormsphere}, and $L^{2}(I; \cdot)$ are Bochner spaces, see \cite[Chap.~5.9.2]{Eva10}.
We refer to the discussion in \cite[Sec.~4.3]{LanSch17} regarding the regularity requirements of $\tilde{\rho}$.

While the above problem is stated in an infinite-dimensional setting, only finitely many (point) evaluations are available in practice.
For each frame $t \in \{ 0, \dots, T \}$, we are given $N_{t} \ge 0$ noisy measurements $\left\{ \tilde{\rho}^{\delta}(t, x_{i}): x_{i} \in \mathcal{S}^{2} \right\}_{i=1}^{N_{t}}$ at pairwise distinct points on $\mathcal{S}^{2}$.
Approximate locations of cell centres serve as measurements, cf. Sec.~\ref{sec:preprocessing}.
Due to the form \eqref{eq:param}, the values of the point evaluations are given by
\begin{equation}
	\tilde{\rho}^{\delta}(t, \bar{x}_{i}) = \norm{x_{i}}, \; x_{i} \in \R^{3} \setminus \{ 0 \}, \; t \in \{ 0, \dots, T \}, \; i \in \{ 1, \dots, N_{t} \},
\label{eq:surfmeasurements}
\end{equation}
where $\bar{x}_{i} = x_{i} / \norm{x_{i}}$, akin to \eqref{eq:extension}, is the radial projection onto the 2-sphere.
See also Fig.~\ref{fig:surfaces} for illustration.
In total, at least one sample point is required.

We attempt to approximate the solution to $\min_{\tilde{\rho}} \mathcal{G}(\tilde{\rho})$ in a finite-dimensional subspace $\mathcal{Q} \subset H^{r}(\mathcal{S}^{2})$.
We choose this space as
\begin{equation*}
	\mathcal{Q} \coloneqq \mathrm{span} \left\{ \tilde{Y}_{p}: p \in J_{\mathcal{Q}} \right\},
\end{equation*}
where $J_{\mathcal{Q}} \subset \N$ again is an index set and $\tilde{Y}_{p}$ are scalar spherical harmonics, see Sec.~\ref{sec:sphericalharmonics}.
For a time instant $t \in \{ 0, \dots, T \}$, the sought function is thus expanded as
\begin{equation}
	\tilde{\rho}(t, \cdot) = \sum_{p \in J_{\mathcal{Q}}} \varrho_{p}(t) \tilde{Y}_{p},
\label{eq:rhoansatz}
\end{equation}
where $\varrho_{p}(t) \in \R$, for $p \in J_{\mathcal{Q}}$, are the time-dependent, unknown coefficients.

With ansatz \eqref{eq:rhoansatz} we find that
\begin{equation*}
	\partial_{t} \tilde{\rho}(t, \cdot) = \partial_{t} \sum_{p \in J_{\mathcal{Q}}} \varrho_{p}(t) \tilde{Y}_{p} = \sum_{p \in J_{\mathcal{Q}}} \partial_{t} \varrho_{p}(t) \tilde{Y}_{p}.
\end{equation*}
Thus, for the last term in \eqref{eq:surffunctional} we have
\begin{equation*}
\begin{aligned}
	\norm{\partial_{t} \tilde{\rho}(t, \cdot)}_{L^{2}(\mathcal{S}^{2})}^{2} & = \norm{\sum_{p \in J_{\mathcal{Q}}} \partial_{t} \varrho_{p}(t) \tilde{Y}_{p}}_{L^{2}(\mathcal{S}^{2})}^{2} \\
	& = \int_{\mathcal{S}^{2}} \Bigl( \sum_{p \in J_{\mathcal{Q}}} \partial_{t} \varrho_{p}(t) \tilde{Y}_{p} \Bigr)^{2} \; d\mathcal{S}^{2} \\
	& = \int_{\mathcal{S}^{2}} \sum_{p \in J_{\mathcal{Q}}} \sum_{q \in J_{\mathcal{Q}}} \partial_{t} \varrho_{p}(t) \partial_{t} \varrho_{q}(t) \tilde{Y}_{p} \tilde{Y}_{q} \; d\mathcal{S}^{2} \\
	& = \sum_{p \in J_{\mathcal{Q}}} \sum_{q \in J_{\mathcal{Q}}} \partial_{t} \varrho_{p}(t) \partial_{t} \varrho_{q}(t) \int_{\mathcal{S}^{2}} \tilde{Y}_{p} \tilde{Y}_{q} \; d\mathcal{S}^{2} \\
	& \overset{\mathclap{\eqref{eq:spharmorthonormal}}}{=} \sum_{p \in J_{\mathcal{Q}}} \bigl( \partial_{t} \varrho_{p}(t) \bigr)^{2}.
\end{aligned}
\end{equation*}
In further consequence, we replace the partial derivatives with respect to time with the backward difference $\partial_{t} \varrho_{p}(t) \coloneqq \varrho_{p}(t) - \varrho_{p}(t - 1)$.

By plugging \eqref{eq:rhoansatz} into \eqref{eq:surffunctional}, utilising definition \eqref{eq:sobolevseminormsphere}, and taking $\partial \mathcal{G} / \partial \varrho_{p}(t)$ for all $p \in J_{\mathcal{Q}}$ and all $t \in \{ 0, \dots, T \}$, we obtain the linear system of optimality conditions
\begin{equation}
\begin{aligned}
	\sum_{q \in J_{\mathcal{Q}}} \varrho_{q}(t) \left( \sum_{i=1}^{N_{t}} \tilde{Y}_{p}(\bar{x}_{i}) \tilde{Y}_{q}(\bar{x}_{i}) \right) & + (\beta_{0} \lambda_{p}^{r} + 2 \beta_{1}) \varrho_{p}(t) \\
	- \beta_{1} \varrho_{p}(t-1) & - \beta_{1} \varrho_{p}(t + 1) = \sum_{i=1}^{N_{t}} \norm{x_{i}} \tilde{Y}_{p}(\bar{x}_{i}), \quad p \in J_{\mathcal{Q}},
\end{aligned}
\label{eq:surfoptcond}
\end{equation}
and enforce (temporal) zero Neumann boundary conditions at $t = 0$ and $t = T$.

\subsection{Evaluation of Integrals} \label{sec:eval}

In order to solve the linear systems \eqref{eq:linsystemof} and \eqref{eq:linsystemcm}, it remains to discuss the numerical evaluation of the involved integrals and the construction of the set \eqref{eq:vecbasisfunset} of basis functions.
For each time instant $t \in I$ we treat this problem in a unified manner on the 2-sphere by utilising identity \eqref{eq:surfintegral} together with a suitable cubature rule.
Given $M$ evaluation points $x_{i} \in \mathcal{S}^{2}$ and corresponding weights $q_{i} \in \R$, we approximate the surface integral of a function $\hat{f}: \mathcal{M}_{t} \to \R$ by
\begin{equation*}
	\int_{\mathcal{M}_{t}} \hat{f} \; d\mathcal{M}_{t} = \int_{\mathcal{S}^{2}} \tilde{f} \tilde{\rho} \sqrt{\norm{\nabla_{\mathcal{S}^{2}} \tilde{\rho}}^{2} + \tilde{\rho}^{2}} \; d\mathcal{S}^{2} \approx \sum_{i=1}^{M} \bigl( \tilde{f} \tilde{\rho} \sqrt{\norm{\nabla_{\mathcal{S}^{2}} \tilde{\rho}}^{2} + \tilde{\rho}^{2}} \bigr)(x_{i}) q_{i},
\end{equation*}
where $\tilde{f}: \mathcal{S}^{2} \to \R$ is as defined in \eqref{eq:coordrepr}.

Since the data motivating this article are supported only on the upper hemisphere, we assume that the coefficients of vectorial basis functions centred at $x_{j}^{3} < 0$ are zero, cf. \eqref{eq:zonalfunc}.
As a result, the number of unknowns in the linear systems \eqref{eq:linsystemof} and \eqref{eq:linsystemcm} is halved.
Moreover, we choose a cubature rule for integration over the spherical cap
\begin{equation*}
	\mathcal{C} \coloneqq \left\{ x \in \mathcal{S}^{2}: \arccos(x \cdot \mathbf{e}_{3}) \le \pi/2 \right\} \subset \mathcal{S}^{2},
\end{equation*}
where $\mathbf{e}_{3} = (0, 0, 1)^{\top} \in \R^{3}$ is the unit vector pointing in $x^{3}$-direction.
We refer to \cite[Sec.~7.1]{HesSloWom10} for more details and the construction of this cubature rule.

To achieve an approximately uniform placement of basis functions \eqref{eq:vecbasisfunset} on the upper hemisphere, we generate a polyhedral approximation $\mathcal{S}_{h}^{2} = (\mathcal{V}, \mathcal{T})$ of $\mathcal{S}^{2}$.
Here, $\mathcal{V} = \{ v_{1}, \dots, v_{n} \} \subset \mathcal{S}^{2}$ is the set of vertices and $\mathcal{T}$ the set of triangular faces.
This triangular mesh is generated by iterative refinement of an icosahedron which is inscribed in the sphere, see e.g. \cite[Chapter~1.3.3]{BotKobPauAllLev10}.
In every refinement step the edge lengths are halved by connecting the edge midpoints and projecting them onto the unit sphere.
The number of vertices of $\mathcal{S}_{h}^{2}$ in iteration $\ell \in \N_{0}$ is $n = 2 + 10 \cdot 4^{\ell}$.
For the placement of basis functions \eqref{eq:vecbasisfunset} we choose $\mathcal{V} \cap \mathcal{C}$ as centre points, resulting in approximately $n$ basis functions, as every point in this set gives rise to two basis functions, cf. \eqref{eq:vecbasisfun}.
\section{Experiments} \label{sec:experiments}

\subsection{Microscopy Data}

The data at hand are volumetric time-lapse (4-dimensional) images of a living zebrafish embryo.
They were recorded with the help of confocal laser-scanning microscopy during the gastula period of the animal, taking place approximately five to ten hours after its fertilisation.
The sequence features endodermal cells which have been labelled with a green fluorescence protein and can therefore be observed separately from the background.
We refer to \cite{MegFra03} for the imaging techniques and the data acquisition, and to \cite{NaiSchi08} for information about the treatment of the specimen.

The recorded microscopy data contains a cuboid region of approximately $860 \times 860 \times 320 \, \mu \mathrm{m}^{3}$ at a spatial resolution of $512 \times 512 \times 44$ voxels.
It features the animal hemisphere and exhibits noise contamination.
Image intensities are in the range $\{0, \dots, 255\}$.
A representative sequence contains 151 frames recorded at a temporal interval of $120 \, \mathrm{s}$.
For further consideration we denote the recorded data by $f^{\delta} \in \{0, \dots, 255\}^{151 \times 512 \times 512 \times 44}$.
See Fig.~\ref{fig:raw} for the unprocessed and noisy microscopy data.

\subsection{Preprocessing and Surface Data Acquisition} \label{sec:preprocessing}

In this section, we briefly outline how we extract an image sequence $\hat{f}$ together with the time-evolving sphere-like surface $\mathcal{M}$ from said microscopy data.
As outlined in Sec.~\ref{sec:surfparam}, we use the approximate centres of cell nuclei as sample points to find the surface.
They represent local maxima in image intensity and can be found with sufficient accuracy by Gaussian filtering each frame $f^{\delta}(t, \cdot)$ followed by thresholding.
However, before solving the surface interpolation problem \eqref{eq:surffunctional}, the points are centred around the origin by fitting one single sphere to the union of all thresholded local maxima and subsequently subtracting the spherical centre.
Then, measurements \eqref{eq:surfmeasurements} are computed and the system \eqref{eq:surfoptcond} of optimality conditions is solved.
After having found a finite-dimensional approximation \eqref{eq:rhoansatz} of $\tilde{\rho}$, all surface quantities derived in Sec.~\ref{sec:surfaces} can be computed.

It remains to discuss the numerical approximation of the image sequence $\hat{f}$, its partial derivative $\partial_{t} \hat{f}$, and the surface gradient $\nabla_{\mathcal{M}} \hat{f}$.
For each frame $t \in \{0, \dots, T\}$, we obtain surface data $\hat{f}(t, x)$ at $x \in \mathcal{M}_{t}$ via the radial projection
\begin{equation}
	\hat{f}(t, x) \coloneqq \max_{c \in [1-\varepsilon, 1+\varepsilon]} \mathring{f}^{\delta}(t, cx),
\label{eq:radialprojection}
\end{equation}
where $\varepsilon > 0$ is chosen sufficiently large.
By $\mathring{f}^{\delta}$ we denote the piecewise linear extension of $f^{\delta}$ to $\R^{3}$, which is required for gridded data.
Before doing so, the intensities $\mathring{f}^{\delta}$ are scaled to the interval $[0, 1]$.
The above projection \eqref{eq:radialprojection} selects the maximum fluorescence-intensity within a narrow band around $\mathcal{M}_{t}$ and thereby allows for small deviations of the cell nuclei from the fitted surface.
See Figs.~\ref{fig:data3} and \ref{fig:cross} for illustration.

Furthermore, we approximate the surface gradient of $\hat{f}$ as
\begin{equation}
	\nabla_{\mathcal{M}}\hat{f}(t, x) \coloneqq \mathrm{P}_{\mathcal{M}}(t, x) \left[ \mean\limits_{c \in [1-\varepsilon, 1+\varepsilon]} \nabla_{\R^{3}} \mathring{f}^{\delta}(t, cx) \right],
\label{eq:surfgradapprox}
\end{equation}
where $\mathrm{P}_{\mathcal{M}}$ is the orthogonal projector defined in \eqref{eq:orthogonalprojector}.
Here, $\nabla_{\R^{3}} \mathring{f}^{\delta}$ is approximated by central differences inside the cuboid and by one-sided differences at the boundaries.
We stress that in the above projection the parameter $\varepsilon$ must be chosen with care as zero values influence the magnitude of the projection \eqref{eq:surfgradapprox} and, in further consequence, the estimated velocity fields.

Finally, for $t = \{ 0, \dots, T-1 \}$, the partial derivative of $\hat{f}$ with respect to time is estimated by the forward difference $\partial_{t} \hat{f}(t, \cdot) \coloneqq \hat{f}(t + 1, \cdot) - \hat{f}(t, \cdot)$.

\subsection{Visualisation of Results} \label{sec:experiments:visualisation}

We utilise the standard flow colour-coding for the visualisation of vector fields \cite{BakSchaLewRotBla11}.
The idea is to create a colour image representation of a (planar) vector field by assigning each vector a colour and an intensity value from a pre-defined colour disk.
The colour and the intensity associated with a vector are determined by its angle, respectively its length.
Typically, the radius $R$ of the colour disk is chosen to be equal to the length of the longest vector in the vector field one attempts to visualise.

In \cite{KirLanSch15, LanSch17}, the idea has been extended to illustrate vector fields on surfaces.
However, before assigning a---not necessarily tangent---vector a colour and an intensity, it is projected to the plane and then scaled to its original length, provided that the length of the projection is non-zero.
Let us denote by $\mathrm{P}_{x^{3}}: (x^{1}, x^{2}, x^{3})^{\top} \to (x^{1}, x^{2}, 0)^{\top}$ the orthogonal projector of $\R^{3}$ onto the $x^{1}$-$x^{2}$-plane.
Then, for a general surface vector field $\mathbf{\hat{X}}(t, \cdot): \mathcal{M}_{t} \to \R^{3}$, we apply the colour-coding to the scaled projection
\begin{equation*}
	\mathbf{\hat{X}} \mapsto
	\begin{cases}
		\frac{\norm{\mathbf{\hat{X}}}}{\norm{\mathrm{P}_{x^{3}} \mathbf{\hat{X}}}} \mathrm{P}_{x^{3}} \mathbf{\hat{X}} & \text{if } \norm{\mathrm{P}_{x^{3}} \mathbf{\hat{X}}} > 0, \\
		0 & \text{else},
	\end{cases}
\end{equation*}
and map the resulting colour image back onto the surface.
As a result, the length of the individual vectors is preserved, provided that they do not vanish in the projection.
We assume that each $\mathcal{M}_{t}$ is such that $\mathrm{P}_{x^{3}}$ is injective.
Moreover, due to the assumptions made in Sec.~\ref{sec:eval}, we only consider visualising the northern hemisphere.

In order to evaluate the computed velocity fields, we create another triangular mesh $\mathcal{S}_{h'}^{2}$ similar to the one in Sec.~\ref{sec:eval}.
Vector fields are then evaluated at the centroids of the triangular faces and thus yield piecewise constant colour-coded images.
For plotting purposes, the surface data is evaluated at the vertices of $\mathcal{S}_{h'}^{2}$ and interpolated piecewise linearly.
Moreover, to simplify matters we plot piecewise linear approximation of the surfaces.
We found that $\ell = 7$ iterative refinement steps sufficiently resolve the microscopy data.

In addition, we illustrate surface velocity fields with streamlines, see e.g. \cite{WeiErl05}.
However, before doing so the velocity fields are projected onto the $x^{1}$-$x^{2}$-plane.
Then, given a steady vector field $\mathbf{v}$ in the plane and a starting point $x_{0}$, a streamline $\gamma(\cdot, x_{0})$ solves the ordinary differential equation
\begin{equation}
\begin{aligned}
	\partial_{\tau} \gamma(\tau, x_{0}) & = \mathbf{v}(\gamma(\tau, x_{0})), \\
	\gamma(0, x_{0}) & = x_{0}.
\end{aligned}
\label{eq:streamline}
\end{equation}
We compute numerical approximations $\gamma_{\kappa}$ of \eqref{eq:streamline} by solving
\begin{equation*}
\begin{aligned}
	\gamma_{\kappa}(\tau + 1, x_{0}) & = \gamma_{\kappa}(\tau, x_{0}) + \kappa \mathbf{v}(\gamma_{\kappa}(\tau, x_{0})), \\
	\gamma_{\kappa}(0, x_{0}) & = x_{0},
\end{aligned}
\end{equation*}
for a given number of initial points $x_{0} \in \R^{2}$ and for $\tau = 50$ iterations.
Here, $\kappa > 0$ is a step size parameter that is set in dependence of $\mathbf{v}$.
Moreover, we apply linear interpolation of $\mathbf{v}$.
With increasing $\tau$ we adjust the colour of $\gamma_{\kappa}$ from yellow to green, see Figs.~\ref{fig:flow2:detail} and \ref{fig:streamlines}.
In what follows we create for a given surface vector field $\mathbf{\hat{X}}$ a streamline visualisation of its projection $\mathrm{P}_{x^{3}} \mathbf{\hat{X}}$ onto the $x^{1}$-$x^{2}$-plane.
Note that, other than in the colour-coding above, we do not rescale the projected vectors.

\subsection{Results}

We conducted several experiments with said zebrafish microscopy data.
In a first step, we minimised functional \eqref{eq:surffunctional} by solving the optimality conditions \eqref{eq:surfoptcond} to obtain an approximation of the deforming surface.
Approximate cell centres were used as sample points of the surface, see the discussion in Sec.~\ref{sec:preprocessing}.
We chose the parameter $r$ of the Sobolev space $H^{r}(\mathcal{S}^{2})$ as $r = 3 + \epsilon$, where $\epsilon = 2.2204 \cdot 10^{-16}$ is the machine precision.
This particular choice originates from regularity requirements discussed in \cite[Sec.~4.3]{LanSch17}.
The regularisation parameters were set to $\beta_{0} = 10^{-4}$ and $\beta_{1} = 100$, and the finite-dimensional subspace $\mathcal{Q} \subset H^{r}(\mathcal{S}^{2})$ was chosen as
\begin{equation*}
	\mathcal{Q} = \mathrm{span} \left\{ \tilde{Y}_{n, j}: n = 0, \dots, 10, j = 1, \dots, 2n + 1 \right\}.
\end{equation*}

\begin{figure}[t]
	\includegraphics[width=0.32\textwidth]{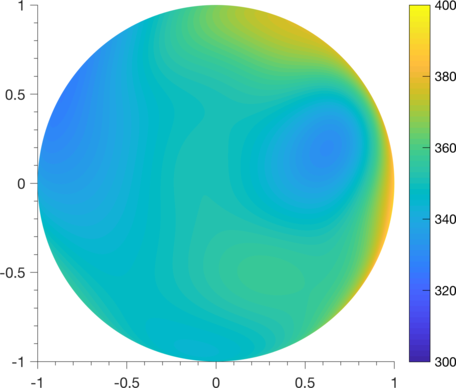} \hfill
	\includegraphics[width=0.32\textwidth]{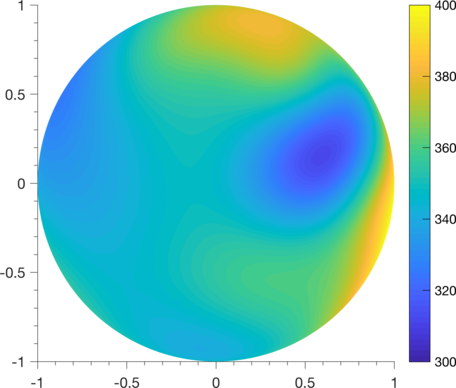} \hfill
	\includegraphics[width=0.32\textwidth]{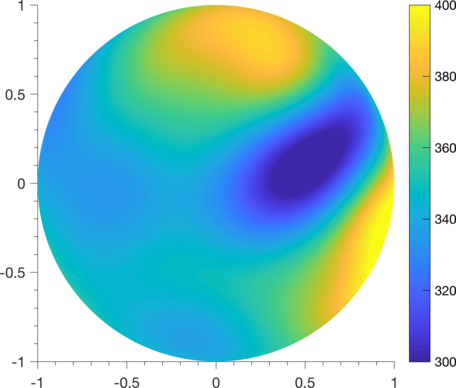}
	\caption{Depicted is a top view of the recovered (radius) function $\tilde{\rho}$ at times 110, 130, and 150 (from left to right). All dimensions are in micrometer ($\mu$m).}
	\label{fig:rho2}
\end{figure}

\begin{figure}[t]
	\includegraphics[width=0.32\textwidth]{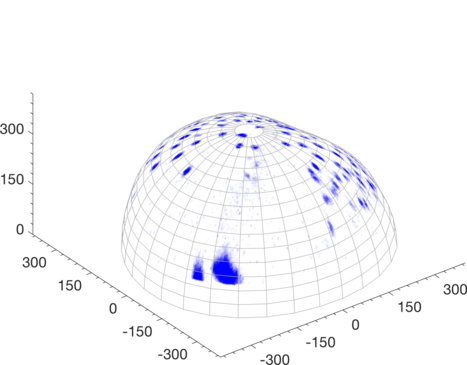} \hfill
	\includegraphics[width=0.32\textwidth]{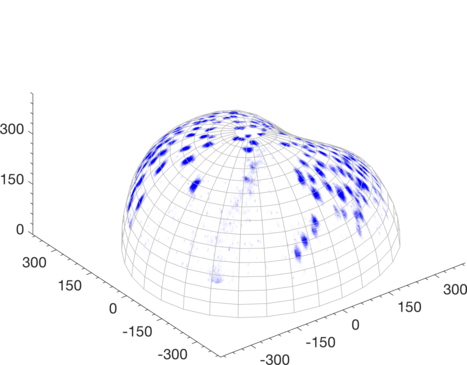} \hfill
	\includegraphics[width=0.32\textwidth]{\resfolder/data3-cxcr4aMO2_290112-grid-150}
	\caption{Shown are frames 110, 130, and 150 (from left to right) of the upper hemisphere of the estimated sphere-like surface $\mathcal{M}_{t}$ together with surface data $\hat{f}$. The curved surface is indicated by an artificial mesh which---for illustration purposes---has been widened in radial direction by one percent of its original distance from the origin. All dimensions are in micrometer ($\mu$m).}
	\label{fig:data3}
\end{figure}

Figure~\ref{fig:rho2} depicts three selected frames of a minimising function $\tilde{\rho}$ of $\mathcal{G}$ computed for frames $\{ 100, 101, \dots, 151 \}$ of the microscopy sequence.
Figure~\ref{fig:data3} shows the estimated surface $\mathcal{M}_{t}$ for the same frames together with the surface data $\hat{f}$, which is obtained by the radial maximum-intensity projection \eqref{eq:radialprojection} with $\varepsilon = 0.1$.
We highlight that the deformation of the embryo is well-captured and contains the anticipated cell features, cf. also the unprocessed volumetric data in Fig.~\ref{fig:raw}.
The growing dent in the surface corresponds to the clearly visible dark blue area in Fig.~\ref{fig:rho2}.

\begin{figure}[t]
	\includegraphics[width=0.32\textwidth]{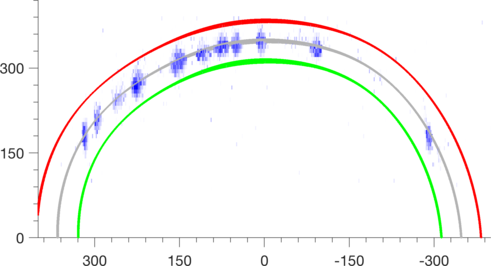} \hfill
	\includegraphics[width=0.32\textwidth]{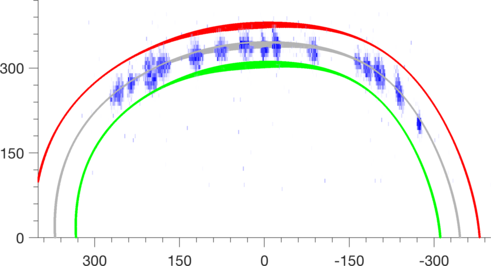} \hfill
	\includegraphics[width=0.32\textwidth]{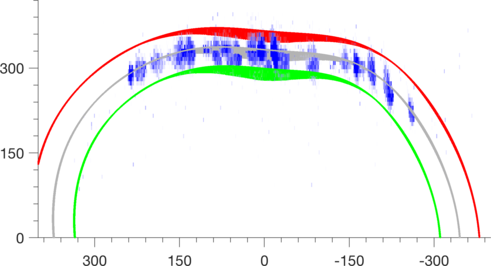}
	\caption{Shown are cross sections of 50~$\mu$m thickness of the unprocessed microscopy data $f^{\delta}$, the interpolated surface $\mathcal{M}_{t}$ (in grey), and the narrow band (red and green) within projection \eqref{eq:radialprojection} is taken to obtain $\hat{f}$. The images correspond to the ones depicted in Fig.~\ref{fig:data3}, i.e. frames 110, 130, and 150 (from left to right). All dimensions are in micrometer ($\mu$m).}
	\label{fig:cross}
\end{figure}

Moreover, Fig.~\ref{fig:cross} illustrates a section of the unprocessed microscopy data $f^{\delta}$ together with the fitted surface (in grey) and the narrow band (in red and green) used in \eqref{eq:radialprojection} and \eqref{eq:surfgradapprox} to obtain the surface data $\hat{f}$ and the surface gradient $\nabla_{\mathcal{M}} \hat{f}$ from the volumetric data $f^{\delta}$, respectively.
Note that the fitted surface accurately represents the single-cell layer and cell material is located almost entirely within the narrow band.

In a second step, we computed minimisers of the functionals \eqref{eq:offunctional} and \eqref{eq:cmfunctional} for one pair of frames by solving the corresponding optimality conditions \eqref{eq:linsystemof} and \eqref{eq:linsystemcm}, respectively.
As outlined in Sec.~\ref{sec:eval}, the finite-dimensional subspace $\mathcal{U} \subset H^{1}(\mathcal{M}_{t}, T\mathcal{M}_{t})$ in \eqref{eq:vecbasisfunspace} was created by five mesh refinements resulting in approximately $N = 10^{4}$ (tangent) vectorial basis functions.
Moreover, the parameters of the basis functions were set to $k = 3$ and $h = 0.99$, cf. Sec.~\ref{sec:background:basisfun}.
The degree of the numerical cubature was chosen as 400, yielding approximately 8600 evaluation points on the spherical cap.
It remains to discuss the choice of the function $s$ in the regularisation functionals \eqref{eq:regfunctional} and \eqref{eq:regfunctional2}.
We chose it in dependence of the surface data $\hat{f}$ as
\begin{equation}
	s \coloneqq \begin{cases}
		1 - \eta & \text{if $1 - \eta < \hat{f}$}, \\
		\hat{f} & \text{if $\eta \le \hat{f} \le 1 - \eta$}, \\
		\eta & \text{if $\hat{f} < \eta$},
	\end{cases}
\label{eq:s}
\end{equation}
with $\eta = 10^{-4}$, which guarantees that $s(t, \cdot): \mathcal{M}_{t} \to (0, 1)$.

All experiments were performed on an Intel Core i5-6500 $3.20 \, \mathrm{GHz}$ MacBook Pro equipped with $16 \, \mathrm{GB}$ RAM.
The running time was governed by the evaluation of the integrals in \eqref{eq:linsystemof} and \eqref{eq:linsystemcm}, which altogether amounts to approximately 150 seconds per pair of frames in our Matlab implementation.
In comparison to previous works \cite{KirLanSch14, LanSch17}, where globally supported vectorial basis functions were employed and computation time was several hours, this represents a significant speed-up.
Furthermore, the memory requirements have been reduced drastically.

All systems of linear equations were solved by application of the backslash operator in Matlab and resulted in a relative residual less than $10^{-14}$ within just a few seconds.
Both the microscopy data\footnote{\url{https://doi.org/10.5281/zenodo.1211599}} and the source code of the implementation\footnote{\url{https://doi.org/10.5281/zenodo.1238910}} are available online.

\begin{figure}[t]
	\includegraphics[width=0.49\textwidth]{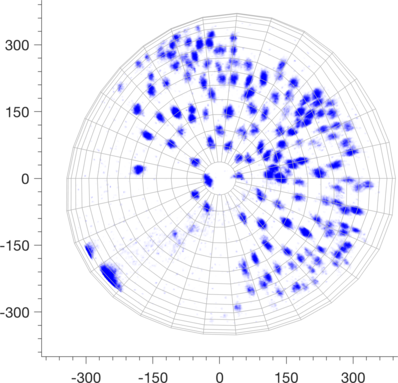} \hfill
	\includegraphics[width=0.49\textwidth]{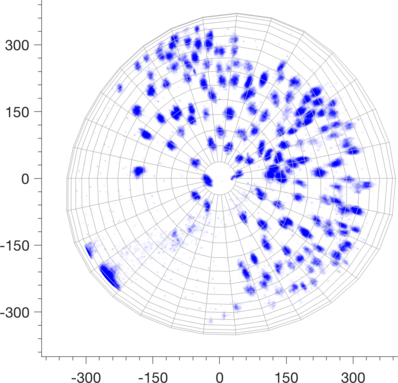}
	\caption{Shown are frames no. 112 (left) and 113 (right) of the processed microscopy image sequence in a top view with an artificial mesh superimposed. The mesh has been widened by one percent of its radius for better illustration. All dimensions are in micrometer ($\mu$m).}
	\label{fig:data2}
\end{figure}

\begin{figure}[t]
	\includegraphics[width=0.49\textwidth]{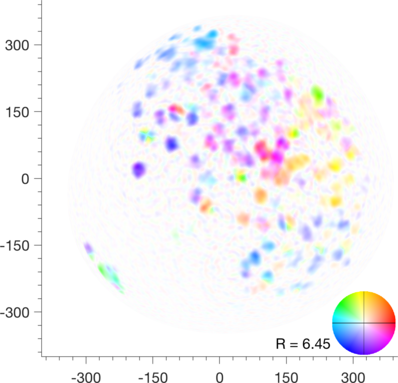} \hfill
	\includegraphics[width=0.49\textwidth]{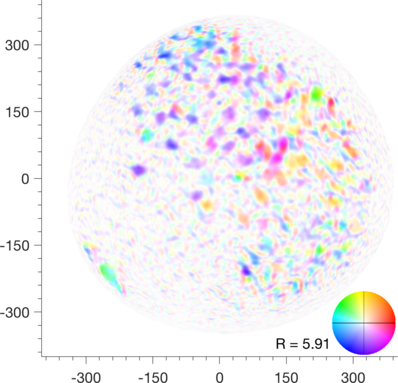} \vspace{0.5em}
	\\
	\includegraphics[width=0.49\textwidth]{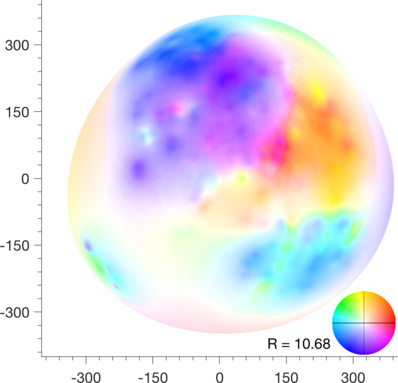} \hfill
	\includegraphics[width=0.49\textwidth]{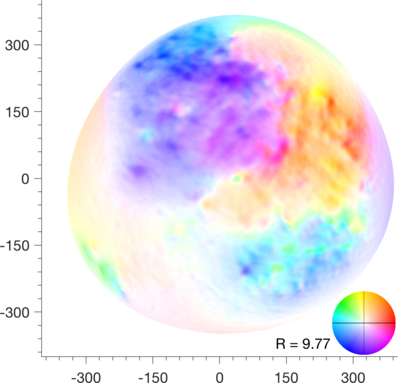}
	\caption{Colour-coded visualisation of minimisers $\mathbf{\hat{w}}$ of $\mathcal{E}$ (left) and $\mathbf{\hat{u}}$  of $\mathcal{F}$ (right), respectively, for two different choices of $s$. The top row shows $s$ as defined in \eqref{eq:s}, while the bottom row illustrates the choice $s \equiv 1$. The regularisation parameters were chosen as follows. Top left: $\alpha_{0} = 10^{-1}$ and $\alpha_{1} = 10^{-3}$. Top right: $\alpha_{0} = 10^{-1}$, $\alpha_{1} = 10^{-3}$, and $\alpha_{2} = 10^{-3}$. Bottom left and bottom right: $\alpha_{0} = 10^{-1}$. All dimensions are in micrometer ($\mu$m).}
	\label{fig:flow2}
\end{figure}

Figure~\ref{fig:data2} depicts the two selected (consecutive) frames of the processed microscopy image sequence in a top view.
All results shown in the following were computed for this particular pair of frames and are also shown in a top view only.

In Fig.~\ref{fig:flow2} we portray minimisers $\mathbf{\hat{w}}$ and $\mathbf{\hat{u}}$ of functionals $\mathcal{E}$ and $\mathcal{F}$, respectively.
Moreover, we compare the effect of two different choices of $s$.
The velocities are visualised with the help of the colour-coding introduced in Sec.~\ref{sec:experiments:visualisation}.
While the top row shows results for $s$ chosen as in \eqref{eq:s} and indicates that individual motion of cells is captured particularly well by the proposed model, the bottom row depicts results for $s \equiv 1$, which provides a better insight into the collective motion of cells on a global scale.
We highlight also the difference in the magnitude of the recovered velocity fields, which is indicated by the radius $R$ of the (scaled) colour disk.

\begin{figure}[t]
	\frame{\includegraphics[width=0.32\textwidth]{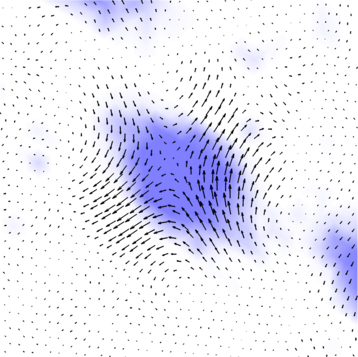}} \hfill
	\frame{\includegraphics[width=0.32\textwidth]{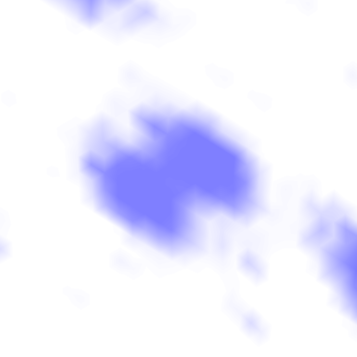}} \hfill
	\frame{\includegraphics[width=0.32\textwidth]{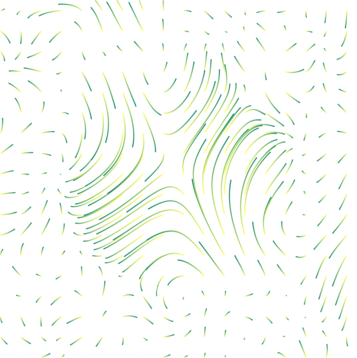}} \vspace{0.5em}
	\\
	\frame{\includegraphics[width=0.32\textwidth]{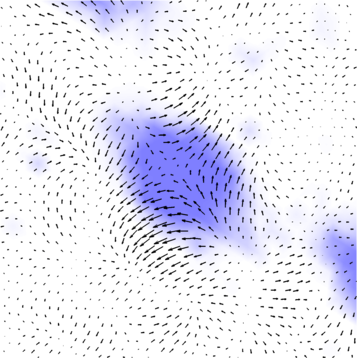}} \hfill
	\frame{\includegraphics[width=0.32\textwidth]{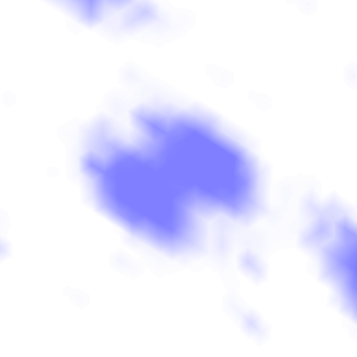}} \hfill
	\frame{\includegraphics[width=0.32\textwidth]{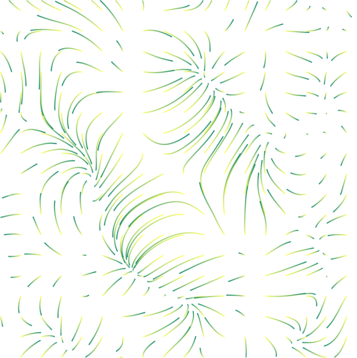}}
	\caption{Detailed view of estimated velocities $\mathbf{\hat{w}}$ (top) and $\mathbf{\hat{u}}$ (bottom) during a cell division. Depicted are the surface data $\hat{f}$ at frame $t = 112$ with the respective velocity superimposed (left), the data $\hat{f}$ at frame $t = 113$ (middle), and the streamline representation of the respective velocity (right), as outlined in Sec.~\ref{sec:experiments:visualisation}. The same parameters as in Fig.~\ref{fig:flow2} (top row) were used. For better illustration $\hat{f}$ has been brightened slightly.}
	\label{fig:flow2:detail}
\end{figure}

Figure~\ref{fig:flow2:detail} depicts a detailed section of the velocity fields shown in Fig.~\ref{fig:flow2} (top row) during a cell division.
Clearly, the cell division is adequately captured and the velocity fields are spatially confined.
Notice also the differences in the streamline plot.

\begin{figure}[t]
	\includegraphics[width=0.32\textwidth]{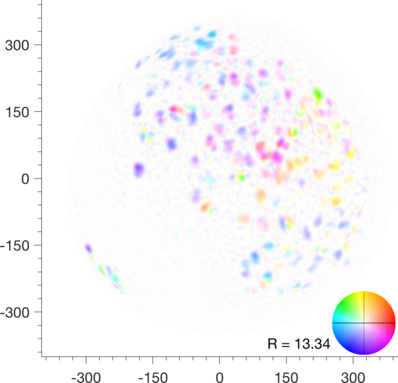} \hfill
	\includegraphics[width=0.32\textwidth]{{\resfolder/of-flow2-cxcr4aMO2_290112-alpha-0.1-beta-0.001-112}.png} \hfill
	\includegraphics[width=0.32\textwidth]{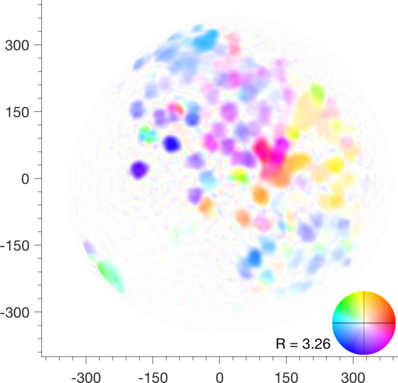} \vspace{0.5em}
	\\
	\includegraphics[width=0.32\textwidth]{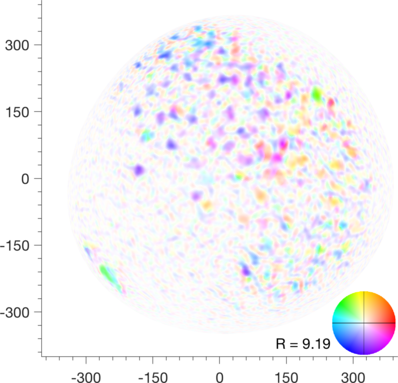} \hfill
	\includegraphics[width=0.32\textwidth]{{\resfolder/cm-flow2-cxcr4aMO2_290112-alpha-0.1-beta-0.001-gamma-0.001-112}.png} \hfill
	\includegraphics[width=0.32\textwidth]{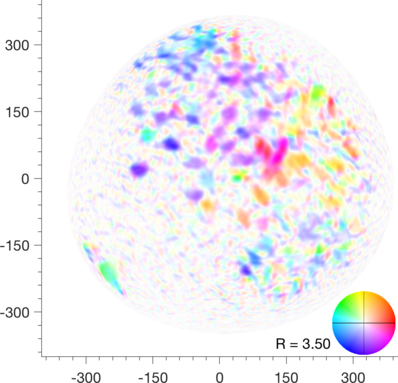}
	\caption{Computed tangent vector fields $\mathbf{\hat{w}}$ (top) and $\mathbf{\hat{u}}$ (bottom) for increasing regularisation parameter $\alpha_{0} = 10^{-2}$ (left), $\alpha_{0} = 10^{-1}$ (middle), and $\alpha_{0} = 1$ (right). The other parameters were kept fixed as $\alpha_{1} = 10^{-3}$ and $\alpha_{2} = 10^{-3}$. The function $s$ was set as in \eqref{eq:s}.}
	\label{fig:flow2:regparam1}
\end{figure}

In Fig.~\ref{fig:flow2:regparam1}, we illustrate tangent vector fields obtained for increasing regularisation parameter $\alpha_{0}$ for $s$ chosen as in \eqref{eq:s}.
Observe in both rows the broadening of the support and the decrease in magnitude of the velocity fields for increasing $\alpha_{0}$.

\begin{figure}[t]
	\includegraphics[width=0.32\textwidth]{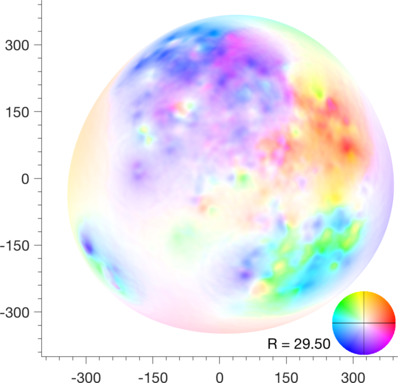} \hfill
	\includegraphics[width=0.32\textwidth]{{\resfolder/of-flow2-one-cxcr4aMO2_290112-alpha-0.1-beta-0-112}.png} \hfill
	\includegraphics[width=0.32\textwidth]{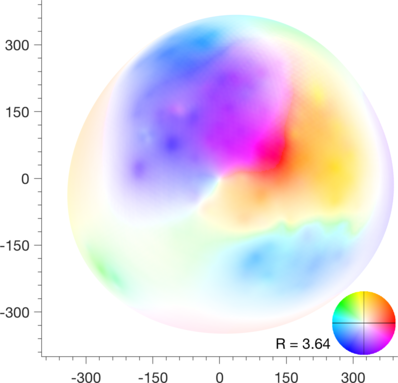} \vspace{0.5em}
	\\
	\includegraphics[width=0.32\textwidth]{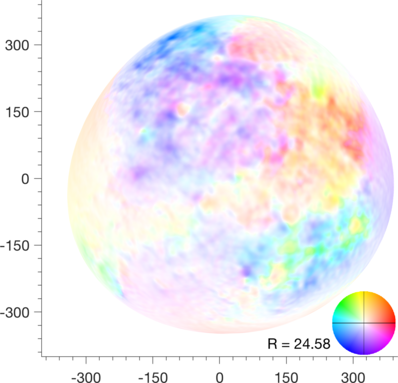} \hfill
	\includegraphics[width=0.32\textwidth]{{\resfolder/cm-flow2-one-cxcr4aMO2_290112-alpha-0.1-beta-0-gamma-0-112}.png} \hfill
	\includegraphics[width=0.32\textwidth]{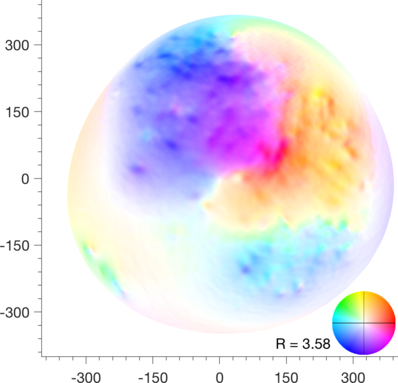}
	\caption{Computed tangent vector fields $\mathbf{\hat{w}}$ (top) and $\mathbf{\hat{u}}$ (bottom) for increasing regularisation parameter $\alpha_{0} = 10^{-2}$ (left), $\alpha_{0} = 10^{-1}$ (middle), and $\alpha_{0} = 1$ (right). The other parameters were kept fixed as $\alpha_{1} = 10^{-3}$ and $\alpha_{2} = 10^{-3}$. The function $s$ was set to $s \equiv 1$.}
	\label{fig:flow2:regparam2}
\end{figure}

\begin{figure}[t]
	\includegraphics[width=0.32\textwidth]{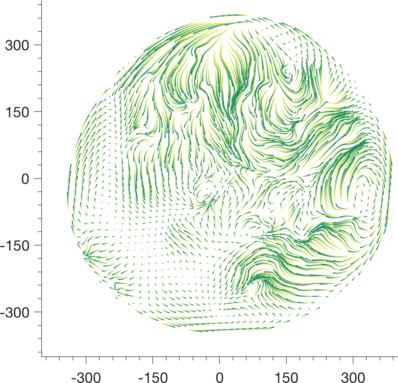} \hfill
	\includegraphics[width=0.32\textwidth]{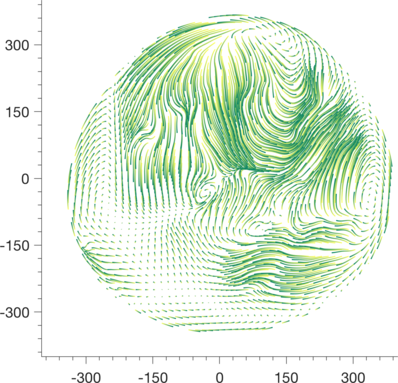} \hfill
	\includegraphics[width=0.32\textwidth]{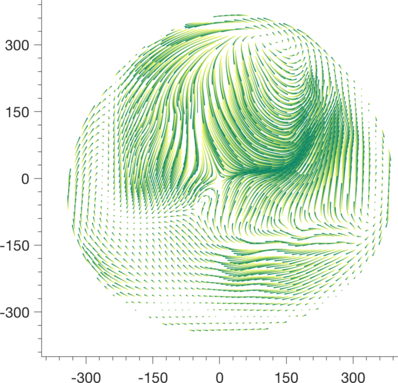} \vspace{0.5em}
	\\
	\includegraphics[width=0.32\textwidth]{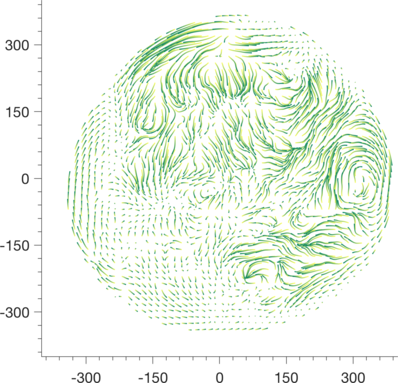} \hfill
	\includegraphics[width=0.32\textwidth]{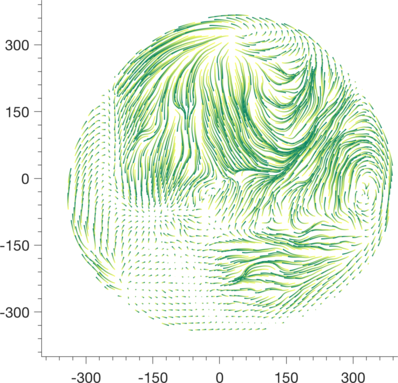} \hfill
	\includegraphics[width=0.32\textwidth]{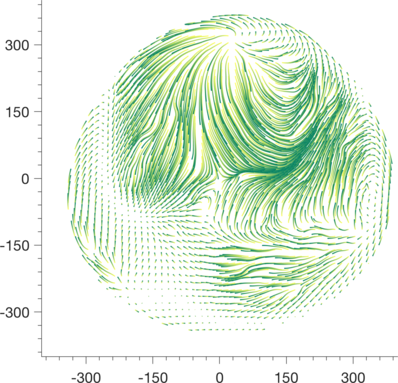}
	\caption{Shown are streamlines visualising the velocity fields from Fig.~\ref{fig:flow2:regparam2}, i.e. velocity fields obtained for increasing regularisation parameter $\alpha_{0}$ (from left to right). The top row shows streamlines computed for $\mathrm{P}_{x^{3}} \mathbf{\hat{w}}$, while the bottom row illustrates it for $\mathrm{P}_{x^{3}} \mathbf{\hat{u}}$. The change in colour from yellow to green illustrates the increasing parameter $\tau$.}
	\label{fig:streamlines}
\end{figure}

Moreover, Fig.~\ref{fig:flow2:regparam2} depicts minimisers for increasing parameter $\alpha_{0}$ for $s \equiv 1$.
As expected, the velocity fields become more regular with increasing $\alpha_{0}$ and, due to the regularisation functional being a norm (see Sec.~\ref{sec:background:sobolevspaces}), decrease in magnitude, which again is indicated by the radius $R$ of the colour disk.
These findings are in line with the results obtained in \cite{KirLanSch14, LanSch17}.
In Fig.~\ref{fig:streamlines}, we visualise the velocity fields from Fig.~\ref{fig:flow2:regparam2} with the help of streamlines as outlined in Sec.~\ref{sec:experiments:visualisation}.

\begin{figure}[t]
	\includegraphics[width=0.32\textwidth]{{\resfolder/of-flow2-cxcr4aMO2_290112-alpha-0.1-beta-0.001-112}.png} \hfill
	\includegraphics[width=0.32\textwidth]{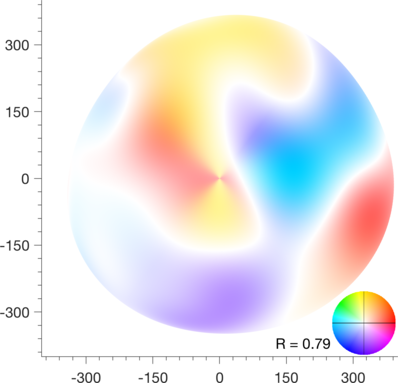} \hfill
	\includegraphics[width=0.32\textwidth]{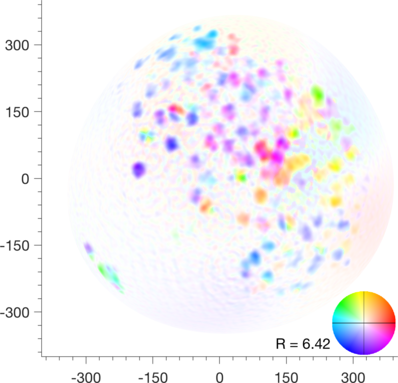} \vspace{0.5em}
	\\
	\includegraphics[width=0.32\textwidth]{{\resfolder/cm-flow2-cxcr4aMO2_290112-alpha-0.1-beta-0.001-gamma-0.001-112}.png} \hfill
	\includegraphics[width=0.32\textwidth]{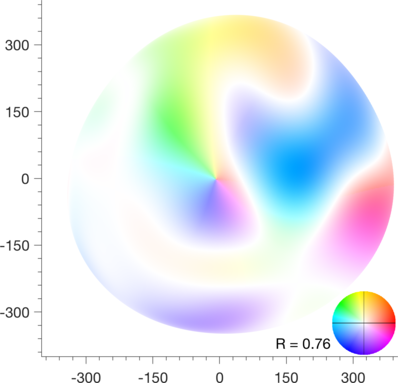} \hfill
	\includegraphics[width=0.32\textwidth]{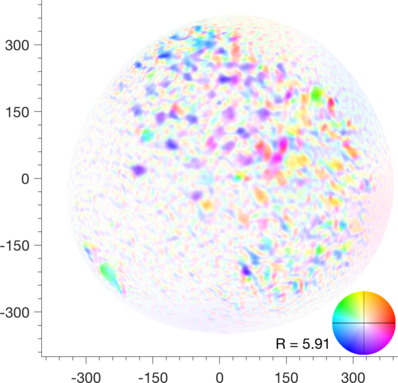}
	\caption{The top row depicts the velocity $\mathbf{\hat{w}}$ (left), the surface velocity $\mathbf{\hat{V}}$ (middle), and the total velocity estimated as $\mathbf{\hat{V}} + \mathbf{\hat{w}}$ (right). The bottom row depicts the velocity $\mathbf{\hat{u}}$ (left), the normal component $V \mathbf{\hat{N}}$ (middle), and the total velocity estimated as $V \mathbf{\hat{N}} + \mathbf{\hat{u}}$ (right). The same parameters as in Fig.~\ref{fig:flow2} (top row) were used.}
	\label{fig:motion2}
\end{figure}

As the main purpose of this article is cell motion estimation in volumetric fluorescence microscopy data, we also illustrate the total velocity $\mathbf{\hat{U}}$ of cells.
Recall from Sec.~\ref{sec:model} that for estimated velocity fields $\mathbf{\hat{w}}$ and $\mathbf{\hat{u}}$, it can be computed as $\mathbf{\hat{U}} = \mathbf{\hat{V}} + \mathbf{\hat{w}}$, respectively as $\mathbf{\hat{U}} = V \mathbf{\hat{N}} + \mathbf{\hat{u}}$.
Here, $\mathbf{\hat{V}}$ denotes the surface velocity, $V$ its (scalar) normal component, and $\mathbf{\hat{N}}$ the outward unit normal.
From the definition of the parametrisation $\y$ in \eqref{eq:param} it follows that the surface velocity is given by
\begin{equation*}
	\mathbf{\hat{V}} = \partial_{t} \y = \partial_{t} \tilde{\rho} \x.
\end{equation*}
Note that $\mathbf{\hat{V}}$ is radial.
In Fig.~\ref{fig:motion2}, we compare the total velocities obtained for the velocity fields shown in Fig.~\ref{fig:flow2} (top row).
Observe the difference between $\mathbf{\hat{V}}$ and $V \mathbf{\hat{N}}$, and the significant difference in the magnitudes of the visualised velocities.
\section{Conclusion} \label{sec:conclusion}

With the intention of efficient motion estimation in volumetric microscopy data of a living zebrafish embryo, we followed the paradigm of dimensional reduction and considered brightness and mass conservation on evolving sphere-like surfaces.
We derived a generalised continuity equation valid for time-varying surfaces embedded in Euclidean 3-space and discussed its relation to the generalised optical flow equation derived in \cite{KirLanSch15}.
In light of the ill-posedness of the discussed conservation laws we proposed the use of spatially varying regularisation functionals suited for considered microscopy data of fluorescently labelled cells.
For the efficient numerical solution we devised a Galerkin method based on compactly supported (tangent) vectorial basis functions allowing for efficient evaluation of the optimality conditions.
A significant performance improvement compared to previous methods that are based on globally supported basis functions was observed.
In order to accurately estimate the velocity of the (artificially) imposed sphere-like surface, we considered surface interpolation with spatial and temporal regularisation, which can be approximately and efficiently minimised with the help of scalar spherical harmonics expansion.
We performed several experiments on the basis of aforementioned zebrafish microscopy data
The computed velocity fields indicate that cell motion can be estimated well and efficiently with the proposed method.

\paragraph{Acknowledgements}
The author thanks Pia Aanstad for kindly providing the microscopy data.
Moreover, he is grateful to Peter Elbau, Christian Gerhards, and Clemens Kirisits for their helpful comments.
The author acknowledges support from Leverhulme Trust project ``Breaking the non-convexity barrier'', EPSRC grant ``EP/M00483X/1'', EPSRC centre ``EP/N014588/1'', the Cantab Capital Institute for the Mathematics of Information, and from CHiPS (Horizon 2020 RISE project grant).


\def\cprime{$'$} \providecommand{\noopsort}[1]{}

\end{document}